\documentclass[A4paper,11pt,leqno]{article}
\usepackage{amsmath,amssymb,amsfonts,amsthm}
\usepackage{enumerate}
\usepackage{enumitem}
\usepackage{graphics} 
\usepackage{float}
\usepackage{epsfig} 
\usepackage{graphicx}
\usepackage{epstopdf}
\usepackage{hyperref}
\usepackage{framed}
\usepackage{verbatim}

\usepackage{bm}
\usepackage{color}
\usepackage[a4paper,left=2.2cm, right=2.2cm,top=2.4cm, bottom=2.4cm]{geometry}
\usepackage{natbib}
\usepackage[a4paper]{geometry}
\usepackage{bbm}
\newtheorem{theorem}{Theorem}[section]
\newtheorem{corollary}[theorem]{Corollary}

\newtheorem{lemma}[theorem]{Lemma}
\newtheorem{proposition}[theorem]{Proposition}

\theoremstyle{definition}
\newtheorem{definition}[theorem]{Definition}
\newtheorem{remark}[theorem]{Remark}

\newcommand{\eps}{\varepsilon}

\newcommand{\mattia}[1]{{\color{black}{#1}}}

\newcommand{\Lip}{\textup{Lip}}

\newcommand{\N}{\mathbb{N}}

\newcommand{\R}{\mathbb{R}}

\newcommand{\diver}{\textup{div}}

\newcommand{\Id}{\textup{Id}}

\newcommand{\defin}{:=}

\newcommand{\Hoelder}[3]{\mathcal{C}^{#1,#2}(#3)}
\newcommand{\mylabel}[2]{#2\def\@currentlabel{#2}\label{#1}}

\newcommand{\say}[1]{``#1''}

\newenvironment{claim}[1]{\par\noindent\textit{Claim:}\space#1}{}
\newenvironment{claimproof}[1]{\par\noindent\textit{Proof of Claim:}\space#1}{\leavevmode\unskip\penalty9999 \hbox{}\nobreak\hfill\quad\hbox{$\blacksquare$}}

\renewenvironment{framed}[1][\hsize]
{\MakeFramed{\hsize#1\advance\hsize-\width \FrameRestore}}%
{\endMakeFramed}

\allowdisplaybreaks

\title{Mean Field Games of Controls with Dirichlet boundary conditions}

\begin{document}

\author{Mattia Bongini\footnote{Advanced Analytics Business Line, CRIF S.p.A.,
		Via M. Fantin 3, Bologna, IT-40131, Italy. (e-mail: m.bongini@crif.com)} \ \& Francesco Salvarani\footnote{
L\'eonard de Vinci P\^ole Universitaire, Research Center,
92916 Paris La D\'efense, France \& 
Dipartimento di Matematica ``F. Casorati'',
Universit\`a degli Studi di Pavia,
Via Ferrata 1, 27100 Pavia, Italy. (e-mail: francesco.salvarani@unipv.it)}}

\maketitle

\begin{abstract}
	In this paper we study a mean-field games system with Dirichlet boundary conditions in a closed domain and in a \textit{mean-field \mattia{game of controls}} setting, that is in which the dynamics of each agent is affected not only by the average position of the rest of the agents but also by their average optimal choice. This setting allows the modeling of more realistic real-life scenarios in which agents not only will leave the domain at a certain point in time (like during the evacuation of pedestrians or in debt refinancing dynamics) but also act competitively to anticipate the strategies of the other agents.
	
	\noindent
	We shall establish the existence of Nash Equilibria for such class of mean-field \mattia{game} of controls systems under certain regularity assumptions on the dynamics and the Lagrangian cost. Much of the paper is devoted to establishing several a priori estimates which are needed to circumvent the fact that the mass is not conserved (as we are in a Dirichlet boundary condition setting). In the conclusive sections, we provide examples of systems falling into our framework as well as numerical implementations.
\end{abstract}

\tableofcontents

\section{Introduction}\label{sec:intro}

\paragraph{Context and literature review.} The study of dynamical systems of interacting agents \mattia{has become extremely popular} in the last years, since their flexibility and tractability has allowed them to faithfully model the main emerging patterns of several biological, social and economic phenomena: from the collective motion of animals \cite{parisi08} and pedestrians \cite{bellomo2011SR, helbing2001RMP}, to the behavior of market participants \cite{cordier2005kinetic} and the emergence of price volatility \cite{ahn13}, passing through the study of optimization heuristics \cite{dorigo2005ant,kennedy2011particle}. For recent developments on their application in the modeling of complex systems, we point to the surveys \cite{bak2013nature,bonginichapsparse,carrillo2017review,viza12} and references therein.

Beside the literature on this topic, another complementary literature has grown in parallel at the same pace: the one on the control of such systems. This literature can be roughly divided into two main streams. The one on \textit{centralized controls} focuses on the modeling approach of an external control agent that exerts its force to steer the system towards a desired optimal configuration. Many papers have studied this framework in the discrete \cite{CFPT12,bonginioriginsparse} as well as in the continuous setting \cite{bongini2016pontryagin, MFOC, piccoli2015control, fornasier2019mean, burger2020instantaneous}, where the framework has come to be known as \textit{mean-field optimal control}.

Alongside this centralized approach to control, there is the stream dedicated to decentralized control strategies: these consist in assuming that each agent, besides being subjected to forces induced in a feedback manner by the rest of the population, follows an individual strategy to coordinate with the other agents based on their position. \textit{Mean-field games} (in short MFG) stem from this control setting when the number of agents is very large, and were first introduced in \cite{lasry2007mean} and \mattia{independently} in the engineering community by \cite{huang2003individual}. For a comparison of the MFG setting and the mean field optimal control one, we refer the interested reader to the book \cite{bensoussan2013mean}.

MFG have provided a powerful framework for the description of large systems of mutually competing agents, such as pedestrian dynamics (for which we highlight to the reader the interesting papers \cite{burger2013mean,santambrogio2011modest}) and market dynamics (a topic which has seen a wealth of literature flourishing since the seminal paper \cite{gueant2011mean}). For these modeling scenarios, it is of particular relevance the study of MFG with Dirichlet boundary conditions, since they arise naturally in situations where agents can leave the domain, such as in evacuation problems, in the case of financial default or in an exhaustible resources framework \cite{graber2015existence}. We mention that, in the paper \cite{ferreira2019existence}, the authors have proven the existence and uniqueness of solutions to a stationary MFG with Dirichlet boundary conditions, while in \cite{MR3906158} the author has studied well-posedness of MFG in the context of optimal stopping (a framework resembling Dirichlet boundary conditions).

However, the MFG framework may not be enough to fully capture the behavior of competing agents since in real-life situations, when choosing a strategy, they usually take into account not only the positions of their competitors but also their strategy. For such reason, a more general class of MFG systems were introduced in \cite{gomes2014existence} under the terminology of \textit{extended MFG models} or \textit{mean-field \mattia{game of controls}} to allow for stochastic dynamics depending upon the joint distribution of the controlled state and the control process. The mean-field \mattia{game of controls} setting allowed the authors of \cite{cardaliaguet2016mean} to analyze optimal liquidation strategies in a context where trader's choices affect the public price of a good, and has since been employed by several authors \cite{acciaio2019extended,alasseur2020extended, fu2021mean}. Existence and uniqueness results for mean-field \mattia{game} of controls on $\mathbb{T}^d$ were discussed in the \mattia{recent} work \cite{kobeissi2019classical}.
 
Beside the recent advances on the theoretical side, a rich literature on numerical methods for MFG has flourished. The first results go back to the pioneering times of MFG theory, such as 
\cite{MR2679575}, then many other developments have been published (such as the articles \cite{MR3097034, MR2888257, MR3575615}).
We refer to \cite{MR4214777} for an up-to-date and more complete picture on the state-of-the-art on the numerics of MFG equations.
However, if we consider the study of mean-field games systems in a mean-field \mattia{game of controls} setting, the literature is much less developed.
In particular, \cite{MR4146720} has proposed a finite difference scheme coupled with an iterative method for solving the systems of nonlinear equations that arise in the discrete setting. Further details on this topic can be found in \cite{lauriere2021numerical}.

\paragraph{Aim of the paper and main contributions.} In this paper we study the following family of mean-field \mattia{game of controls} models with Dirichlet boundary conditions that naturally arise in the context of pedestrian and investor dynamics:

\begin{align}\label{eq:mfg}
\left\{\begin{aligned}
&\partial_t u(t,x) + \sigma \Delta u(t,x) - H(t,x,D u(t,x);\mu_t) = 0, & \text{ for } x\in\Omega \text{ and } t \in[0,T]\\
&\partial_t m_t(x) - \sigma \Delta m_t(x) - \diver(D_p H(t,x,D u(t,x);\mu_t)m_t(x)) = 0, & \text{ for } x\in\Omega \text{ and } t \in[0,T]\\
&\mu_t = (\Id,\alpha^*(t,\cdot, Du(t,\cdot);\mu_t))_{\sharp}m_t, &\text{ for }  t \in[0,T]\\
&m_0(x) = \overline{m}(x), u(T,x) = \psi(x,m_T), & \text{ for } x\in\overline{\Omega} \\
&m_t(x) = 0, u(t,x) = \psi(x,m_t), & \text{ for } x\in\partial\Omega\text{ and } t \in[0,T]
\end{aligned}\right.
\end{align}

Here, $m_t$ denotes the distribution of the agents' states, $u$ is the value function of the optimal control problem (with optimal control $\alpha^*$), and $\mu_t$ stands for the joint distribution of the state and the control. \mattia{The boundary condition for $m$, that is $m=0$ on $\partial\Omega$, indicates that the agents stop taking part to the game as soon as they reach the boundary.} As we shall see in a moment, the above system is obtained from a standard MFG setting as soon as we require that agents need to optimize their strategy by taking into account the strategy of the others, which leads to the fixed-point relationship
$$\mu_t = (\Id,\alpha^*(t,\cdot, Du(t,\cdot);\mu_t))_{\sharp}m_t$$
between the joint distribution $\mu_t$ and its marginals $m_t$ and \mattia{$\alpha^*(t,\cdot, Du(t,\cdot);\mu_t)_{\sharp}m_t$}.

A primary contribution of this paper is that we give sufficient conditions for the existence of Nash Equilibria for such systems (that is, their well-posedness). To prove the result we need to employ a double fixed-point method, in accordance to the formulation of the system. To this end we need to establish several a priori regularity estimates on the dynamics of our system. The Dirichlet boundary condition prevents us to leverage upon the conservation of mass of $m_t$ for such estimates, hence we need to rely on weaker results to work around the issue, and to resort to techniques that, to the best of our knowledge, are novel in the mean-field \mattia{game} of controls framework. The importance of Dirichlet boundary conditions in applications can be appreciated in \cite{MR4146720}, which provides novel numerical schemes to simulate solutions of such problems, although without providing existence results. Moreover, apart from the technical leap of dealing with Dirichlet boundary conditions, we remark that the setting is more general from the one of \cite{kobeissi2019classical} since we also prove our existence result on $\R^d$ (rather than on $\mathbb{T}^d$). Furthermore, our framework allows for richer agents' dynamics like, for instance, the Cucker-Smale and the Hegselmann-Krause models in crowd motion. All these differences introduce several specific technical complications that set the two works apart, as far as technical details and modeling applications are concerned. The uniqueness problem for our modeling framework shall be the subject of a dedicated forthcoming work.

A second contribution of the paper is that we provide two explicit models fitting into our  framework belonging to very different paradigms: one in the context of evacuation problems, the other formulated in a debt refinancing setting. This shows the flexibility of our framework and its potential to be exported to different modeling scenarios.

A third and final contribution is the implementation, through an iterative procedure based on the particle method, of mean-field games system with Dirichlet boundary conditions in a mean-field \mattia{game of controls} setting.
This method draws \mattia{has an analogy} with the results of convergence of large population games to mean field games (see, for example, 
\cite{lauriere2020convergence}).
Up to the best of our knowledge, \mattia{our approach is} one of the first uses of particle methods to numerically solve 
mean-field games system with Dirichlet boundary conditions in a mean-field \mattia{game of controls} setting, in which an iterative procedure is implemented for reducing the control space.

\paragraph{Derivation of the model.} To help the reader in the interpretability of the mean-field \mattia{game} of controls system \eqref{eq:mfg}, we now give a formal derivation from the particle dynamics of an individual agent optimizing its strategy.

We denote by $X_t\in\overline{\Omega}$ the state of an \textit{infinitesimal} agent of the system in the \textit{state space} $\overline{\Omega}$, and by $\alpha_t$ its control, which takes values in the \textit{control space} $A \subset \R^d$. The distribution density of the state-control pair $(X_t,\alpha_t)$ shall be denoted by $\mu_t$, which is a probability measure on the product space $\overline{\Omega}\times A$ whose first marginal $\pi_{1\sharp}\mu_t$ is the distribution of the agents' states $m_t$.

We assume the dynamics of an agent to be
\begin{align*}
\left\{\begin{aligned}
d X_t &= b(t,X_t,\alpha^*_t;\mu_t)dt + 2\sqrt{\sigma}dW_t,\\
X_{\mattia{0}} &= x_0,
\end{aligned}\right.
\end{align*}
where $W_t$ is a $d$-dimensional Brownian motion. Due to the presence of a boundary, the integral form of the dynamics is given by
\begin{align*}
X_t = x_0 + \int_0^{t\wedge\tau} b(s,X_s,\alpha^*_s;\mu_s) ds + 2\sigma\int^{t\wedge\tau}_0 dW_s
\end{align*}
for any $t \mattia{\in (0, \tau)}$, where $\tau$ denotes the \textit{exit time of the agent from $\Omega$}
\begin{align}\label{eq:exitime}
\tau \defin \inf \{t\in[0,T]\mid X_t \in \partial \Omega \text{ or } t = T \}.
\end{align}
The optimal control $\alpha^*_t$ is chosen to solve
\begin{align}\label{eq:costfunmfg}
\min_{\alpha\in \mathcal{U}} J(\mattia{0},x_0,\alpha;\mu) = \mathbb{E}\left[\int^{T\wedge\tau}_0 \mathcal{L}(t,X_t,\alpha_t;\mu_t)dt + \psi(X_{T\wedge\tau},m_{T\wedge\tau})\right].
\end{align}
The control set $\mathcal{U}$ is the set of progressively measurable functions $\alpha:[0,T]\rightarrow A$ with respect to the filtration generated by the process $\{W_t\}_{t\in[0,T]}$. Notice that, since we are in a \textit{mean-field \mattia{game} of controls} setting, the dynamics of each agent is affected not only by the average position of the rest of the agents, given by the distribution $m_t$, but also by their average optimal choice $\alpha^*_t$, whence the dependence of $\mathcal{L}$ on $\mu_t$, instead of just its first marginal as in the usual MFG framework. 

For a given $\mu:[0,T]\rightarrow\mathcal{M}_1(\overline{\Omega}\times A)$, the value function
\begin{align*}
u(\mattia{0},x_0) = \inf_{\alpha\in\mathcal{U}} J(\mattia{0},x_0,\alpha;\mu)
\end{align*}
satisfies the following Hamilton-Jacobi-Bellman equation with Dirichlet boundary conditions
\begin{align}\label{eq:mfghjb}
\left\{\begin{aligned}
&\partial_t u(t,x) + \sigma \Delta u(t,x) - H(t,x,D u(t,x);\mu_t) = 0, & \text{ for } x\in\Omega \text{ and } t \in[0,T]\\
&u(T,x) = \psi(x,m_T) & \text{ for } x\in\overline{\Omega}\\
&u(t,x) = \psi(x,m_t) & \text{ for } x\in\partial\Omega\text{ and } t \in[0,T]
\end{aligned}\right.
\end{align}
where $H$ stands for the Hamiltonian of the system
\begin{align}\label{def:hamiltonian}
H(t,x,p;\nu) \defin \sup_{\alpha \in A}\left\{-p\cdot b(t,x,\alpha;\nu) - \mathcal{L}(t,x,\alpha;\nu) \right\}
\end{align}

If we assume that for every $(t,x,p,\nu)$ there exists a unique $\alpha^*(t,x,p,\nu)$ for which
$$H(t,x,p;\nu) = -p\cdot b(t,x,\alpha^*(t,x,p,\nu);\nu) - \mathcal{L}(t,x,\alpha^*(t,x,p,\nu);\nu),$$
then the optimal control of an infinitesimal agent at time $t$ is given by $\alpha^*(t,X_t, Du(t,X_t);\mu_t)$. This means that, in an equilibrium configuration, the measure $\mu_t$ is the pushforward of $m_t$ by the map $x\mapsto (x,\alpha^*(t,x, Du(t,x);\mu_t))$, i.e., it holds the fixed-point relationship
\begin{align}\label{eq:fixpointmu}
\mu_t = (\Id,\alpha^*(t,\cdot, Du(t,\cdot);\mu_t))_{\sharp}m_t.
\end{align}
This, in particular means that $\mu_t$ is uniquely determined by $Du_t$ and $m_t$, hence for every $t \in [0,T]$ it holds
\begin{align}\label{eq:mufunction}
\mu_t = \Gamma(Du_t,m_t)
\end{align}
for some function $\Gamma$ (which will be investigated thoroughly in Section \ref{sec:gammasection}).

At the same time, the evolution of the mass $m_t$ satisfies a Fokker-Planck type equation with drift given by
$$b(t,x,\alpha^*(t,x,Du(t,x);\mu_t)) = -D_p H(t,x,Du(t,x);\mu_t),$$
which yields the system
\begin{align}\label{eq:fokkerplanck}
\left\{\begin{aligned}
&\partial_t m_t(x) - \sigma \Delta m_t(x) - \diver(D_p H(t,x,D u(t,x);\mu_t)m_t(x)) = 0, & \text{ for } x\in\Omega \text{ and } t \in[0,T]\\
&m_0(x) = \overline{m}(x), & \text{ for } x\in\overline{\Omega} \\
&m_t(x) = 0, & \text{ for } x\in\partial\Omega\text{ and } t \in[0,T]
\end{aligned}\right.
\end{align}


Plugging together the dynamics \eqref{eq:mfghjb} of $u$ and \eqref{eq:fokkerplanck} of $m$, as well as the fixed-point relationship \eqref{eq:fixpointmu} for $\mu_t$, we get the \textit{mean-field games system} \eqref{eq:mfg}.



\paragraph{Structure of the paper.} The paper is organized as follows: in Section \ref{sec:assumption} we shall state the main assumptions under which we will establish the existence of Nash equilibria of system \eqref{eq:mfg}. Section \ref{sec:estimates} contains the main result of the paper, that is the well-posedness of system \eqref{eq:mfg} under the set of assumptions stated in Section \ref{sec:assumption} through a priori estimates and regularity techniques. In Section \ref{sec:examples} we shall provide examples of systems satisfying our set of hypotheses. We conclude the paper with Section \ref{sec:numerics} which reports a numerical implementation of our modeling scenario.

\section{Preliminaries and main assumptions}\label{sec:assumption}

In this section we shall first give some preliminary notations and results that will be instrumental in the following sections. We then move on to list the assumptions under which we shall prove the well-posedness of the mean-field \mattia{game} of controls system \eqref{eq:mfg}.

Let $\Omega\subseteq\R^d$ be an open, \mattia{pre}compact set with smooth boundary $\partial \Omega$, fix $T > 0$ and define $\overline{Q}_T \defin \mattia{[0,T]\times\overline{\Omega}}$. We shall denote by $|\cdot|$ any finite $\R^d$-norm. For any domains $X, Y$, we define by 
\begin{itemize}
	\item $\mathcal{C}^n(X;Y)$ the space of all functions from $X$ to $Y$ which are continuous together with all their derivatives up to order $n$, equipped with the norm $\|\cdot\|_{\mathcal{C}^n(X;Y)}$ (or simply $\|\cdot\|_{\mathcal{C}^n})$ whenever clear from the context);
	\item $\mathcal{C}^\alpha(X;Y)$ for any $\alpha \in (0,1)$ the space of all H\"older
	continuous functions from $X$ to $Y$ with exponent $\alpha$;
	\item $\mathcal{C}^{n+\alpha}(X;Y)$ the set of all functions whose
	$n$ derivatives are all in $\mathcal{C}^\alpha(X;Y)$;
	\item for any subset $X \subseteq Q_T$, $\Hoelder{n +\alpha}{(n+\alpha)/2}{X;Y}$ the set of functions from $X$ to $Y \subseteq \R^n$ which
	are continuous together with all the derivatives $D^r_t D^s_x$ for $2r + s \leq n$ and whose derivatives of order $n$ are H\"older continuous with exponent $\alpha$ in space and exponent $\alpha/2$ in time;
\end{itemize}
As a shorthand notation, we shall write $\Hoelder{\gamma}{\gamma/2}{X;Y}$ in place of $\Hoelder{n +\alpha}{(n+\alpha)/2}{X;Y}$ with the convention that $n = \lfloor\gamma\rfloor$ and $\alpha = n - \gamma$. We denote the $\gamma$-H\"older norm of $\Hoelder{\gamma}{\gamma/2}{X;Y}$ by $\|\cdot\|_{\Hoelder{\gamma}{\gamma/2}{X;Y}}$. We remember that the $\gamma$-H\"older norm is the sum of the supremum norms for derivatives up to $\lfloor\gamma\rfloor$-order, plus the sum of $(n - \gamma)$-H\"older coefficients for the $\lfloor\gamma\rfloor$-order derivatives, for a precise definition we refer the reader to \cite[pages 7-8]{solonnikov}.

Given a compact set $B\subset\R^d$ \mattia{and $\Sigma$ a $\sigma$-algebra over $B$}, we shall denote by $\mathcal{M}_1(B)$ the set of all positive measures from \mattia{$\Sigma$ to $\R$} such that $m(B) \leq 1.$ It is well-known that 
the weak$^*$ convergence in $\mathcal{M}_1(B)$ can be metrized by any of the following metrics (parametrized by \mattia{$a > 1$}):
\begin{align*}
d_{a}(\mu,\nu)\defin \sum^{\infty}_{k = 1} \frac{1}{{a}^{ k}}\frac{\left|\int_{B}f_k(x)d(\mu - \nu)(x)\right|}{1 + \left|\int_{B}f_k(x)d(\mu - \nu)(x)\right|}.
\end{align*}
where, by the Stone-Weierstrass theorem, the set $(f_k)_{k\in\N^+}$ can be obtained by putting together \mattia{any} basis of the vector space of polynomials of degree $n$ over the ring $\R[x_1,\ldots,x_d]$ for every $n\in\N$.

The following technical result shall be helpful in the following.
\begin{lemma}\label{le:techweakconv}
	Let $B\subset \mattia{\R^d}$ be a compact set, \mattia{let
	\begin{align}\label{eq:deltadefin}
	\delta(B) \defin \sup \{|x| \mid x \in B\}
	\end{align}}
	 and let $h:B\rightarrow\R^d$ be a \mattia{continuous} bounded function. Then for every $a>0$ it holds
	$$\left|\int_{B}h(x) d(\nu_1-\nu_2)(x)\right| \leq a(1+2\delta(B))\|h\|_{\infty}  d_a(\nu_1,\nu_2)$$
	holds for every $\nu_1,\nu_2\in\mathcal{M}_1(B)$, where $\|\cdot\|_{\infty}$ denotes the supremum norm.
\end{lemma}
\begin{proof}
	\mattia{First of all, notice that $\delta <+\infty$ since $B$ is compact}. If $(f_k)_{k\in\N^+}$ is a basis of polynomials then each $f_k$ satisfies
	$$\mattia{\left|\int_{B}f_k(x) d(\nu_1-\nu_2)(x)\right| \leq \sup_{x\in B}|f_k(x)| \left(|\nu_1(B)| + |\nu_2(B)|\right) \leq 2 \delta(B)^{\deg(f_k)},}$$
	where we used the fact that $\nu_1,\nu_2\in\mathcal{M}_1(B)$ and we denote by $\deg(f_k)$ the degree of $f_k$. But then, using that $f_1 \equiv 1$ we have
	\begin{align*}
	\left|\int_{B}h(x) d(\nu_1-\nu_2)(x)\right| &\leq \|h\|_{\infty}\left|\int_{B} d(\nu_1-\nu_2)(x)\right|\\
	&= \|h\|_{\infty}a\left(1+\left|\int_{B}f_1(x) d(\nu_1-\nu_2)(x)\right|\right)\frac{1}{a}\frac{\left|\int_{B}f_1(x) d(\nu_1-\nu_2)(x)\right|}{1+\left|\int_{B} f_1(x)d(\nu_1-\nu_2)(x)\right|}\\
	&\leq \|h\|_{\infty}a(1+2\delta(B))\sum^\infty_{k=1}\frac{1}{a^k}\frac{\left|\int_{B}f_k(x) d(\nu_1-\nu_2)(x)\right|}{1+\left|\int_{B}f_k(x) d(\nu_1-\nu_2)(x)\right|}\\
	&\leq a(1+2\delta(B))\|h\|_{\infty} d_a(\nu_1,\nu_2).
	\end{align*}
	This concludes the proof.
\end{proof}

\mattia{
\begin{remark}
	The fact that, for a space of measures on a compact set, the Wasserstein and weak$^*$ convergence are equivalent is well-known (see, e.g., \cite[Thm 5.10]{santambrogiobook}). Lemma \ref{le:techweakconv} provides only a computable bound for such equivalence for a given continuous bounded function $h$, which shall be useful later on.
\end{remark}
}

We now state the necessary hypotheses that the various functions of our extended MFG system need to fulfil.

\begin{framed}[1.1\textwidth]
	\begin{center}
		{\bf Assumptions (A)}
	\end{center}
In the following, fix $\beta > 0$ and a \mattia{continuous} bounded function $\xi:\overline{\Omega}\times A\rightarrow\R^d$. We shall assume that
\begin{enumerate}[label=(A\arabic*)]
	\item\label{item:compactset} The sets $\overline{\Omega}$ and $A$ are compact subsets of $\R^d$, and we denote by \mattia{$\delta \defin \delta(\overline{\Omega}\times A)$, where $\delta$ is defined as in \eqref{eq:deltadefin}}.
	
	\item\label{item:positivem} $\overline{m}\in \mathcal{C}^{2+\beta}(\overline{\Omega};\R)$ such that $\overline{m}(x)\geq0$ for every $x\in\overline{\Omega}$ and
	$$\int_{\overline{\Omega}}\overline{m}(x)dx=1.$$
	
	\item\label{item:hoelderpsi} The map $\psi:[0,T]\times \Hoelder{\beta}{\beta/2}{\overline{Q}_T;\R} \rightarrow \R$ satisfies the following conditions: if $m \in \Hoelder{\beta}{\beta/2}{\overline{Q}_T;\R}$ then
	\begin{enumerate}[label=$(\roman*)$]
		\item\label{item:hoelderpsi1} $\psi(\cdot,m) \in \Hoelder{\beta}{\beta/2}{\overline{Q}_T;\R}$,
		\item\label{item:hoelderpsi2} $\|\psi(\cdot,m)\|_{\Hoelder{\beta}{\beta/2}{\overline{Q}_T;\R}} \leq \mathcal{K}_{\psi}\left(\|m\|_{\Hoelder{\beta}{\beta/2}{\overline{Q}_T;\R}}\right)$ for some continuous function $\mathcal{K}_{\psi}:\R\rightarrow\R$.
	\end{enumerate}

	\item\label{item:inverseb} The map $b:[0,T]\times\overline{\Omega}\times A\times\mathcal{M}_1(\overline{\Omega}\times A)\rightarrow\R^d$
	satisfies the following conditions:
	\begin{enumerate}[label=$(\roman*)$]
		\item\label{item:inverseb2} if $\alpha \in\Hoelder{\beta}{\beta/2}{\overline{Q}_T;\R^d}$ and $\nu_t \in \mathcal{C}^{\beta/2}([0,T];\mathcal{M}_1(\overline{\Omega}\times A))$, then it holds
		$$\mattia{b(t,x,\alpha(t,x);\nu_t)} \in \Hoelder{\beta}{\beta/2}{\overline{Q}_T;\R^d}.$$
	\end{enumerate}

	\item\label{item:monotoneL} The map $\mathcal{L}:[0,T]\times\overline{\Omega}\times A\times\mathcal{M}_1(\overline{\Omega}\times A)\rightarrow\R$ satisfies 
	the following conditions:
	\begin{enumerate}[label=$(\roman*)$]
		\item\label{item:monotoneL1} for every $\nu^1,\nu^2 \in \mathcal{M}_1(\overline{\Omega}\times A)$ and for every $t \in [0,T]$ the monotonicity condition holds:
		\begin{align*}
		\int_{\overline{\Omega}\times A} (\mathcal{L}(t,x,\alpha;\nu^1) - \mathcal{L}(t,x,\alpha;\nu^2))d(\nu^1 - \nu^2)(x,\alpha) \geq 0,
		\end{align*}
		\item\label{item:monotoneL2} if $\alpha \in\Hoelder{\beta}{\beta/2}{\overline{Q}_T;\R^d}$ and $\nu_t \in \mathcal{C}^{\beta/2}([0,T];\mathcal{M}_1(\overline{\Omega}\times A))$, then it holds
		$$ \mathcal{L}(t,x,\alpha(t,x);\nu_t) \in \Hoelder{\beta}{\beta/2}{\overline{Q}_T;\R}.$$		
	\end{enumerate}

	\item\label{item:contalpha} For each $t,x,p,\nu\in[0,T]\times\overline{\Omega}\times\R^d\times\mathcal{M}_1(\overline{\Omega}\times A)$ there exists a unique maximum point $\alpha^*(t,x,p;\nu) \in A$ of
	\begin{align}\label{eq:alphamaximum}
	- p\cdot b(t,x,\alpha^*;\nu) - \mathcal{L}(t,x,\alpha^*;\nu)
	\end{align}
	and the corresponding function $\alpha^*:[0,T]\times\overline{\Omega}\times\R^d\times\mathcal{M}_1(\overline{\Omega}\times A)\rightarrow A$ satisfies
	\begin{enumerate}[label=$(\roman*)$]
		\item\label{item:contalpha1} if $p \in\Hoelder{\beta}{\beta/2}{\overline{Q}_T;\R^d}$ and $\nu_t \in \mathcal{C}^{\beta/2}([0,T];\mathcal{M}_1(\overline{\Omega}\times A))$, then it holds
		$$\alpha(t,x,p(t,x);\nu_t) \in \Hoelder{\beta}{\beta/2}{\overline{Q}_T;A},$$
		\item \label{item:contalpha4} for every $p^1, p^2 \in \R^d$ and $\nu^1, \nu^2 \in \mathcal{M}_1(\overline{\Omega}\times A)$ it holds
		\begin{align*}
		\int_{\overline{\Omega}}|\alpha^*(t,x,p^1;\nu^1)-\alpha^*(t,x,p^2;\nu^2)| dx \leq
		L\left(|p^1 - p^2| + \left|\int_{\overline{\Omega}\times A}\xi(x,\alpha) d(\nu^1-\nu^2)(x,\alpha)\right|\right),
		\end{align*}
		where the constant $L$ satisfies
		\begin{align}\label{eq:alphalipbound}
		L < \frac{1}{\mattia{3(\delta+1)}(1+2\delta)\|\xi\|_{\infty}}.
		\end{align}
	\end{enumerate}

	\item\label{item:hoelderH} The Hamiltonian $H:[0,T]\times \overline{\Omega} \times \R^d \times \mathcal{M}_1(\overline{\Omega}\times A) \rightarrow \R$ defined in \eqref{def:hamiltonian} satisfies the following statement: there exist two continuous functions $\mathcal{H}_1, \mathcal{H}_2:\R\rightarrow\R$ such that, for every
		$(p,m)\in\Hoelder{\beta}{\beta/2}{\overline{Q}_T;\R^d}\times \Hoelder{\beta}{\beta/2}{\overline{Q}_T;\R}$ it holds
		\begin{align*}
		\|H(t,x,p(t,x);\Gamma(p_t,m_t))\|_{\Hoelder{\beta}{\beta/2}{\overline{Q}_T;\R}}
		\leq \mathcal{H}_1\left( \|p\|_{\Hoelder{\beta}{\beta/2}{\overline{Q}_T;\R^d}} + \|m\|_{\Hoelder{\beta}{\beta/2}{\overline{Q}_T;\R}} \right),
		\end{align*}
		as well as
		\begin{align*}
		\|D_pH(t,x,p(t,x);\Gamma(p_t,m_t))\|_{\Hoelder{\beta}{\beta/2}{\overline{Q}_T;\R^d}}
		\leq \mathcal{H}_2\left( \|p\|_{\Hoelder{\beta}{\beta/2}{\overline{Q}_T;\R^d}} + \|m\|_{\Hoelder{\beta}{\beta/2}{\overline{Q}_T;\R}} \right),
		\end{align*}
	where $\Gamma$ is the function defined in \eqref{eq:mufunction}.
\end{enumerate}
\end{framed}

\mattia{
\begin{remark}
	We allow for $\beta$ to be greater than 1 since, from the notation adopted at the beginning of the Section, in that case we refer to the H\"older regularity of the $\lfloor\beta\rfloor$-order derivatives.
\end{remark}}
	
Solutions of system \eqref{eq:mfg} will be interpreted in the classical sense.

\begin{definition} \label{def:solution}
	Fix $\beta > 0$. We define a $\beta$-\textit{classical solution of system} \eqref{eq:mfg} any function $(u,m) \in \Hoelder{2+\beta}{1+\beta/2}{\overline{Q}_T;\R}\mattia{^2}$ which satisfies system \eqref{eq:mfg} in the classical sense.
\end{definition}

What follows is the main theoretical result of the paper, the well-posedness of system \eqref{eq:mfg}.

\begin{theorem}[Existence of Nash equilibria for system \eqref{eq:mfg}] \label{th:mainresult}
	Let Assumptions (A) hold for some constant $\beta > 0$ and some bounded function $\xi:\overline{\Omega}\times A\rightarrow\R^d$. Then there exists a $\beta$-classical solution of system \eqref{eq:mfg}.
\end{theorem}

\section{Existence of Nash equilibria}\label{sec:estimates}

In this section we shall prove that, under the Assumptions (A) stated in the previous section, there exists a $\beta$-classical solution of system \eqref{eq:mfg}.

Throughout this section, we will fix a $\beta > 0$ and a bounded function $\xi:\overline{\Omega}\times A\rightarrow\R^d$, and assume Assumptions (A) hold. We shall start with some preliminary results on the Fokker-Planck equation and the fixed point relationship \eqref{eq:fixpointmu}, and then we shall derive sharper a priori regularity estimates for $u$, $m$ and $\mu$, which will then be instrumental in obtaining the well-posedness of system \eqref{eq:mfg}.

\subsection{Existence and stability of the map $\Gamma$}\label{sec:gammasection}

To establish the existence of the fixed point \eqref{eq:fixpointmu} we will need that $m_t \in \mathcal{M}_1(\overline{\Omega})$, so that we may rely on the weak$^*$ compactness of this space. However, since we consider smooth solutions of the Fokker-Planck equation \eqref{eq:fokkerplanck}, there is no guarantee that this holds true. We shall prove this fact by using a similar technique to the one employed in \cite{graber2015existence}.

\begin{proposition}\label{prop:measfp}
	Assume that $u:\overline{Q}_T\rightarrow\R$ and $\mu:[0,T]\rightarrow\mathcal{M}_1(\overline{\Omega}\times A)$ are such that
	$$D_p H(t,x,D u(t,x);\mu_t) \in L^{\infty}(\overline{Q}_T)$$
	and let $m$ be a smooth solution of \eqref{eq:fokkerplanck} with $u$ and $\mu$. Then for every $t \in [0,T]$ and $x \in \overline{\Omega}$ it holds
	\begin{align*}
	m_t(x) \geq 0 \quad \text{ and } \quad \|m_t\|_{L^1(\Omega)} \leq \|\overline{m}\|_{L^1(\Omega)}.
	\end{align*}
	Therefore, under assumption \ref{item:positivem}, we have $m_t \in \mathcal{M}_1(\overline{\Omega})$ for every $t\in[0,T]$.
\end{proposition}
\begin{proof}
	Let $\delta > 0$ and $\phi_{\delta}:\R\rightarrow\R$ be defined as
	\begin{align*}
	\phi_{\delta}(s) \defin (s^-)^{2+\delta},
	\end{align*}
	where $s^-\defin (|s|-s)/2$. Notice that $\phi_\delta \in \mathcal{C}^2(\R)$ and
	\begin{align*}
	\phi_{\delta}(s) =
	\begin{cases}
	0 & \text{ if } s \geq 0,\\
	(-s)^{2+\delta} & \text{ if } s < 0.
	\end{cases}
	\end{align*}
	Furthermore
	\begin{align*}
	\phi'_{\delta}(s) =
	\begin{cases}
	0 & \text{ if } s \geq 0,\\
	-(2+\delta)(-s)^{1+\delta} & \text{ if } s < 0,
	\end{cases}\quad\text{ and }\quad
	\phi''_{\delta}(s) =
	\begin{cases}
	0 & \text{ if } s \geq 0,\\
	-(2+\delta)(1+\delta)(-s)^{\delta} & \text{ if } s < 0,
	\end{cases}
	\end{align*}
	which implies that for $\delta \rightarrow 0^+$ we have
	\begin{align*}
	\phi_\delta(s)\rightarrow (s^-)^2, \quad \phi'_{\delta}(s) \rightarrow 2s^- \quad \text{ and } \quad \phi''_{\delta}(s) \rightarrow -2\cdot\mathbbm{1}_{\left\{s < 0\right\}}(s)
	\end{align*}
	pointwise. Finally, notice that $\phi'_{\delta}(0) = 0$.
	
	We now multiply the function $\phi'_{\delta}(m_t(x))$ to \eqref{eq:fokkerplanck} and we integrate in time and space to get
	\begin{align}\begin{split}\label{eq:fokkerplanckweak}
	\int^t_0\int_{\Omega}\partial_t m_s(x)&\phi'_{\delta}(m_s(x))dxds \\
	&= \int^t_0\int_{\Omega}\big(\sigma\Delta m_s(x)+\diver(D_p H(s,x,D u(s,x);\mu_s)m_s(x)) \big)\phi'_{\delta}(m_s(x)) dxdt.
	\end{split}\end{align}
	We use integration by parts on both sides of the equation. On the left, we have the identity
	\begin{align*}
	\int^t_0\int_{\Omega}\partial_t m_s(x)\phi'_{\delta}(m_s(x))dxds =  \int_{\Omega}\phi_{\delta}(m_t(x))dx - \int_{\Omega}\phi_{\delta}(\overline{m}(x))dx;
	\end{align*}
	so that now it holds
	\begin{align*}
	\int_{\Omega}&\phi_{\delta}(m_t(x))dx - \int_{\Omega}\phi_{\delta}(\overline{m}(x))dx \\
	&= - \int^t_0\int_{\Omega}\big(\sigma|D m_s(x)|^2 +D_p H(s,x,D u(s,x);\mu_t)m_s(x)\cdot D m_s(x) \big) \phi''_{\delta}(m_s(x)) dxds.
	\end{align*}
	We can now invoke the dominated convergence theorem to pass the limit $\delta\rightarrow 0^+$ inside the integrals in order to get
	\begin{align*}
	\int_{\Omega}&|m^-_t(x)|^2dx - \int_{\Omega}|\overline{m}^-(x)|^2dx \\
	&= - \int^t_0\int_{\Omega}\big(\sigma D m^-_s(x)|^2 +D_p H(s,x,D u(s,x);\mu_s)m^-_s(x)\cdot D m^-_s(x) \big) dxds.
	\end{align*}
	By invoking assumptions \ref{item:positivem} (yielding $\overline{m}^- = 0$) and Young's inequality we get
	\begin{align*}
	\int_{\Omega}|m^-_t(x)|^2dx &\leq - \sigma\int^t_0\int_{\Omega}|D m^-_s(x)|^2dxds + \frac{\sigma}{2}\int^t_0\int_{\Omega}|D m^-_s(x)|^2dxds \\
	&\quad\quad+ \frac{1}{2\sigma}\int^t_0\int_{\Omega}|D_p H(s,x,D u(s,x);\mu_s)m^-_s(x)|^2 dxds.
	\end{align*}
	As a shorthand notation, set $K(t,x) \defin D_p H(t,x,D u(t,x);\mu_t)$. By the assumptions on $u$ and $\mu$ we have that $\|K\|_{L^{\infty}(\overline{Q}_T)}<+\infty$, therefore H\"older inequality and Poincar\'{e} inequality yield the following estimate from above
	\begin{align*}
	\int_{\Omega}|m^-_t(x)|^2dx &\leq - \frac{\sigma}{2}\int^t_0\int_{\Omega}|D m^-_s(x)|^2dxds +\frac{1}{2\sigma}\| K\|^2_{L^{\infty}(\overline{Q}_T)}\int^t_0\int_{\Omega}|m^-_s(x)|^2 dxds\\
	&\leq \left(-\frac{\sigma C_{\Omega,2}}{2}+\frac{1}{2\sigma}\|K\|^2_{L^{\infty}(\overline{Q}_T)}\right)\int^t_0\int_{\Omega}|m^-_s(x)|^2 dxds,
	\end{align*}
	where $C_{\Omega,2}$ denotes Poincar\'{e} constant. Using Gronwall's lemma we finally get $m^-_t(x) = 0$ for every $t \in [0,T]$ and $x\in\overline{\Omega}$ as desired.
	
	To get the inequality $\|m_t\|_{L^1(\Omega)} \leq \|\overline{m}\|_{L^1(\Omega)}$ it suffices to use the function $\phi_\delta(s) = s^{1+\delta}$. Indeed, using this new version of $\phi'_{\delta}(m(t,x))$ in \eqref{eq:fokkerplanckweak} we obtain
	\begin{align*}
	\int_{\Omega}m_t(x)^{1+\delta}dx - \int_{\Omega}\overline{m}(x)^{1+\delta}dx
	&= - \int^t_0\int_{\Omega}\frac{(1+\delta)\delta}{m_s(x)^{1-\delta}}\sigma|D m_s(x)|^2dxds \\
	&\quad \quad -\int^t_0\int_{\Omega}(1+\delta)\delta m_s(x)^{\delta}D_p H(s,x,D u(s,x);\mu_s)\cdot D m_s(x) dxds.
	\end{align*}
	Therefore, passing to the limit $\delta\rightarrow 0^+$ we get
	\begin{align*}
	\int_{\Omega}m_t(x)dx - \int_{\Omega}\overline{m}(x)dx & = \limsup_{\delta \rightarrow 0^+}\left[\int_{\Omega}m_t(x)^{1+\delta}dx - \int_{\Omega}\overline{m}(x)^{1+\delta}dx\right]\\
	&= -\liminf_{\delta\rightarrow 0^+} \int^t_0\int_{\Omega}\frac{(1+\delta)\delta}{m_t(x)^{1-\delta}}\sigma|D m_s(x)|^2dxds \\
	&\quad \quad -\liminf_{\delta\rightarrow 0^+}\int^t_0\int_{\Omega}(1+\delta)\delta m_s(x)^{\delta}D_p H(s,x,D u(s,x);\mu_s)\cdot D m_s(x) dxds\\
	&\leq 0,
	\end{align*}
	having invoked Fatou's lemma for the first term and the dominated convergence theorem for the second term.
\end{proof}

As anticipated, the well-posedness of the fixed-point map \eqref{eq:fixpointmu} shall follow from the compactness properties of the measure space $\mathcal{M}_1$. The strategy of this proof follows \cite{cardaliaguet2016mean}.

\begin{lemma}\label{lem:fixpointmeas}
	Assume that $m \in \mathcal{M}_1(\overline{\Omega}), p \in L^{\infty}(\overline{\Omega};\R^d)$ and let $\alpha^*$ be the unique maximizer of \eqref{eq:alphamaximum}. Then, for every $t \in [0,T]$, the map $\Phi[p,m]:\mathcal{M}_1(\overline{\Omega}\times A)\rightarrow\mathcal{M}_1(\overline{\Omega}\times A)$ defined as
	\begin{align*}
	\Phi[p,m](\mu) \defin (\Id,\alpha^*(t,\cdot,p(\cdot);\mu))_{\sharp}m
	\end{align*}
	admits a unique fixed point $\mu^* = \Gamma(p,m)$.
\end{lemma}
\begin{proof}
	Let us first show that the map $\Phi$ is well-defined, that is $\Phi[p,m](\mu) \in \mathcal{M}_1(\overline{\Omega}\times A)$. We have
	\begin{align*}
	\int_{\overline{\Omega}\times A} d\Phi[p,m](\mu)(x,\alpha) = \int_{\R^d} \mathbbm{1}_{\overline{\Omega}\times A}(x,\alpha^*(t,x,p(x);\mu))m(x)dx = \int_{\overline{\Omega}}m(x)dx \leq 1.
	\end{align*}
	Since $\mathcal{M}_1(\overline{\Omega}\times A)$ is trivially convex and compact with respect to the weak$^*$ convergence of measures (by the Banach-Alaoglu theorem), to show that $\Phi[p,m]$ has a fixed point via the Schauder fixed point theorem we just need to prove that $\Phi[p,m]$ is continuous. To this end, let $(\mu_n)_{n\in\N}$ be a sequence in $\mathcal{M}_1(\overline{\Omega}\times A)$ weak$^*$ converging to $\mu \in \mathcal{M}_1(\overline{\Omega}\times A)$. Then, by using Assumption \ref{item:contalpha}-\ref{item:contalpha4}, for every $\eta\in\mathcal{C}_b(\overline{\Omega}\times A;\R)$ we get
	\begin{align*}
	\lim_{n\rightarrow\infty}\int_{\overline{\Omega}\times A} &\eta(x,\alpha)d\Phi[p,m](\mu_n)(x,\alpha) =\\
	& = \lim_{n\rightarrow\infty}\int_{\overline{\Omega}} \eta(x,\alpha^*(t,x,p(x);\mu_n))m(x)dx\\
	& = \int_{\overline{\Omega}} \eta(x,\alpha^*(t,x,p(x);\mu))m(x)dx\\
	& = \int_{\overline{\Omega}\times A} \eta(x,\alpha)d\Phi[p,m](\mu)(x,\alpha).
	\end{align*}
	This shows existence. To prove uniqueness, assume that $\mu_1$ and $\mu_2$ are two fixed points of $\Phi[p,m]$. Then by assumption \ref{item:monotoneL}-\ref{item:monotoneL1} we get
	\begin{align*}
	0 & \leq\int_{\overline{\Omega}\times A} (\mathcal{L}(t,x,\alpha;\mu_1) - \mathcal{L}(t,x,\alpha;\mu_2))d(\mu_1 - \mu_2)(x,\alpha)\\
	& = \int_{\overline{\Omega}} (\mathcal{L}(t,x,\alpha^*(t,x,p(x);\mu_1);\mu_1) - \mathcal{L}(t,x,\alpha^*(t,x,p(x);\mu_1);\mu_2))m(x)dx \\
	& \quad \quad - \int_{\overline{\Omega}} (\mathcal{L}(t,x,\alpha^*(t,x,p(x);\mu_2);\mu_1) - \mathcal{L}(t,x,\alpha^*(t,x,p(x);\mu_2);\mu_2))m(x)dx.
	\end{align*}
	If we add and subtract the terms $p(x)\cdot b(t,x,\alpha^*(t,x,p(x);\mu_i))$ for $i = 1,2$ and we rearrange the expression, we get
	\begin{align*}
	0 & \leq \int_{\overline{\Omega}} \Big(p(x)\cdot b(t,x,\alpha^*(t,x,p(x);\mu_1)) + \mathcal{L}(t,x,\alpha^*(t,x,p(x);\mu_1);\mu_1) \\
	& \quad \quad \quad \quad- p(x)\cdot b(t,x,\alpha^*(t,x,p(x);\mu_1))- \mathcal{L}(t,x,\alpha^*(t,x,p(x);\mu_1);\mu_2)\Big)m(x)dx \\
	& \quad \quad + \int_{\overline{\Omega}} \Big(p(x)\cdot b(t,x,\alpha^*(t,x,p(x);\mu_2)) + \mathcal{L}(t,x,\alpha^*(t,x,p(x);\mu_2);\mu_2) \\
	&\quad \quad \quad \quad- p(x)\cdot b(t,x,\alpha^*(t,x,p(x);\mu_2)) - \mathcal{L}(t,x,\alpha^*(t,x,p(x);\mu_2);\mu_1)\Big)m(x)dx.
	\end{align*}
	However, since $\alpha^*(t,x,p(x);\mu_i)$ is the maximum of $- p(x)\cdot b(t,x,\alpha;\mu_i) - \mathcal{L}(t,x,\alpha;\mu_i)$ by Assumption \ref{item:contalpha}, then all the integrands are nonpositive, which implies that $\alpha^*(t,x,p(x);\mu_1) = \alpha^*(t,x,p(x);\mu_2)$ $m$-almost everywhere, which in turn implies $\Phi[p,m](\mu_1) = \Phi[p,m](\mu_2)$, and hence $\mu_1 = \mu_2$.
\end{proof}

Next we show the stability of the fixed-point map $\Gamma.$

\begin{lemma}\label{lem:stability}
	Let $(m_n)_{n\in\N}$ be a sequence weakly$^*$ converging in $\mathcal{M}_1(\overline{\Omega})$ to $m$, and let $(p_n)_{n\in\N}$ be a sequence converging a.e. in $L^{\infty}(\overline{\Omega};\R^d)$. Then $\Gamma(p_n,m_n)$ converges weakly$^*$ in $\mathcal{M}_1(\overline{\Omega}\times A)$ to $\Gamma(p,m)$.
\end{lemma}
\begin{proof}
	Denote by $\mu_n \defin \Gamma(p_n,m_n)$. By definition, $\mu_n = \Phi(p_n,m_n)(\mu_n)$ and, since $\mathcal{M}_1(\overline{\Omega}\times A)$ is weakly$^*$ compact, up to subsequences $(\mu_n)_{n\in\N}$ weak$^*$ converges to some $\mu$. We need to show that 
	$\mu = \Phi(p,m)(\mu)$.
	
	Notice that, by Assumption \ref{item:contalpha}-\ref{item:contalpha4} for every $\eta \in \mathcal{C}_b(\overline{\Omega}\times A;\R)$ we have
	\begin{align*}
	\lim_{n\rightarrow\infty}\eta(x,\alpha^*(t,x,p_n(x);\mu_n)) = \eta(x,\alpha^*(t,x,p(x);\mu))
	\end{align*}
	Therefore, if we invoke the dominated convergence theorem we get
	\begin{align*}
	\lim_{n\rightarrow\infty}\int_{\overline{\Omega}\times A}\eta(x,\alpha)d\Phi[p_n,m_n](\Gamma(p_n,m_n))(x,\alpha)&=\lim_{n\rightarrow\infty}\int_{\overline{\Omega}}\eta(x,\alpha^*(t,x,p_n(x);\Gamma(p_n,m_n)))m_n(x)dx \\
	&=\int_{\overline{\Omega}}\eta(x,\alpha^*(t,x,p(x);\mu))m(x)dx\\
	& = \int_{\overline{\Omega}\times A}\eta(x,\alpha)d\Phi[p,m](\mu)(x,\alpha),
	\end{align*}
	which implies that $(\Phi[p_n,m_n](\mu_n))_{n\in\N}$ weak$^*$ converges to $\Phi[p,m](\mu)$. But the fixed-point relationship $\mu_n = \Phi(p_n,m_n)(\mu_n)$ and the uniqueness of the weak$^*$ limit then imply $\mu = \Phi(p,m)(\mu)$, as desired.
\end{proof}

\subsection{A priori regularity estimates}

We shall now extend Lemma \ref{lem:fixpointmeas} to H\"older continuous measure-valued curves.

\begin{lemma}\label{lem:hoelderfixpointmeas}
	Let $P_0,M_0>0$ and let $p \in \Hoelder{\beta}{\beta/2}{[0,T];\R^d})$ and $m \in \Hoelder{\beta}{\beta/2}{[0,T];\R}$ be such that
	$$\|m\|_{\Hoelder{\beta}{\beta/2}{\overline{Q}_T;\R}} \leq M_0 \quad \text{ and } \quad \|p\|_{\Hoelder{\beta}{\beta/2}{\overline{Q}_T;\R^d}} \leq P_0,$$
	with $m_t\in\mathcal{M}_1(\overline{\Omega})$ for every $t\in[0,T]$.
	
	Then there exists a $C_0 \defin C_0(M_0,P_0) > 0$ and a metric on $\mathcal{M}_1(\overline{\Omega}\times A)$ that metrizes the weak$^*$ convergence such that the application $\Psi$ from the set
	\begin{align}\label{eq:compactsetmu}
	\left\{\mu \in \mathcal{C}^{\beta/2}([0,T];\mathcal{M}_1(\overline{\Omega}\times A)) \mid \|\mu\|_{\mathcal{C}^{\beta/2}([0,T];\mathcal{M}_1(\overline{\Omega}\times A))}\leq C_0\right\}
	\end{align}
	onto itself defined as
	\begin{align*}
	\Psi[p,m](\mu) \defin (\Phi[p_t,m_t](\mu_t))_{t \in [0,T]}
	\end{align*}
	has the measure-valued curve $(\Gamma(p_t,m_t))_{t\in[0,T]}$ as unique fixed point. In particular, $(\Gamma(p_t,m_t))_{t\in[0,T]}$ is H\"older continuous in time with exponent $\beta/2$ with respect to such metric.
\end{lemma}
\begin{proof}
	As already argued in Section \ref{sec:assumption}, the metric
	\begin{align*}
	\lambda_{\omega}(\mu,\nu) \defin \sum^{\infty}_{k = 1} \frac{1}{{(\delta+1)}^{\omega k}}\frac{\left|\int_{\overline{\Omega}\times A}f_k(x,\alpha)d(\mu - \nu)(x,\alpha)\right|}{1 + \left|\int_{\overline{\Omega}\times A}f_k(x,\alpha)d(\mu - \nu)(x,\alpha)\right|}.
	\end{align*}
	metrizes the weak$^*$ convergence over $\mathcal{M}_1(\overline{\Omega}\times A)$, where $(f_k)_{k\in\N^+}$ is a family of polynomials such that the degree of $f_k$ is precisely
	$$\deg(k) = \sup \left\{n\in\N\mid\sum^n_{i = 0}\binom{i+d-1}{d-1} \leq k \right\}.$$
	Notice the trivial estimate $\deg(k)\leq k$ for any $k\in\N^+$. We shall prove that, for a proper choice of $C_0$ and $\omega$, the map $\Psi$ defines a contraction over the set \eqref{eq:compactsetmu} (where each $\mu$ is H\"older continuous with respect to the metric $\lambda_\omega$) by using the fact that
	\begin{align*}
	\lambda_{\omega}(\mu,\nu) \leq \sum^{\infty}_{k = 1} \frac{1}{{(\delta+1)}^{\omega k}}\left|\int_{\overline{\Omega}\times A}f_k(x,\alpha)d(\mu - \nu)(x,\alpha)\right|,
	\end{align*}
	and obtaining a bound from above for the right-hand side term.
	
	To choose $C_0$ and $\omega$, we start by first choosing $a > 0$ such that
	$$3(\delta+1)e < a < \frac{e}{La(1+2\delta)\|\xi\|_{\infty}},$$
	(which exists by the bound \eqref{eq:alphalipbound}). \mattia{The left-hand side of the inequality trivially implies
	\begin{align}\label{eq:a1}
	\log(a) > \log(3(\delta + 1))
	\end{align}
	and
	\begin{align}\label{eq:a2}
	a > 3(\delta + 1) e > e \Longrightarrow \log(a) > 1.
	\end{align}
	The right-hand side of the inequality instead implies
	\begin{align*}
	\frac{e}{L(1+2\delta)\|\xi\|_{\infty}} > a \Longrightarrow 1 > \frac{La(1+2\delta)\|\xi\|_{\infty}}{e}
	\end{align*}
	which together with \eqref{eq:a2} yields
	\begin{align}\label{eq:a3}
	\log(a) > \frac{La(1+2\delta)\|\xi\|_{\infty}}{e}.
	\end{align}
	Putting together \eqref{eq:a1} and \eqref{eq:a3} we arrive at}
	$$\log(a) > \log(3(\delta+1)) + \frac{La(1+2\delta)\|\xi\|_{\infty}}{e}.$$
	Now denote by $L_1$ the H\"older constant in \ref{item:contalpha}-\ref{item:contalpha1}. 
	Since for every $M_0, P_0 > 0$ it holds
	$$\lim_{C\rightarrow+\infty} \frac{L_1 + LP_0 + M_0 + La(1+2\delta)\|\xi\|_{\infty}C}{e C} = \frac{La(1+2\delta)\|\xi\|_{\infty}}{e},$$
	choose $C_0 \defin C(M_0,P_0)$ large enough such that
	$$\log(a) > \log(3(\delta+1)) + \frac{L_1 + LP_0 + M_0 + La(1+2\delta)\|\xi\|_{\infty}C_0}{e C_0}$$
	still holds true. It is easy to check that
	$$\frac{L_1 + LP_0 + M_0 + La(1+2\delta)\|\xi\|_{\infty}C_0}{e C_0} \geq \sup_{k\in\N^+}\frac{\log\left(\frac{k}{C_0}(L_1 + LP_0 + M_0 + La(1+2\delta)\|\xi\|_{\infty}C_0)\right)}{k},$$
	so that it actualy holds
	$$\log(a) > \log(3(\delta+1)) + \sup_{k\in\N^+}\frac{\log\left(\frac{k}{C_0}(L_1 + LP_0 + M_0 + La(1+2\delta)\|\xi\|_{\infty}C_0)\right)}{k}.$$
	We can therefore fix $\omega \defin \omega(M_0,P_0) > 0$ satisfying
	\begin{align}\label{eq:omegachoice}
	\frac{\log(a)}{\log(\delta+1)} > \omega > \sup_{k\in\N^+}\frac{k\log(3(\delta+1)) +\log\left(\frac{k}{C_0}(L_1 + LP_0 + M_0 + La(1+2\delta)\|\xi\|_{\infty}C_0)\right)}{k\log(\delta+1)}.
	\end{align}
	
	We therefore pass to bound from above the distance $\lambda_\omega$. Denote by $\overline{\mu}_t \defin \Phi[p_t,m_t](\mu_t)$, whose existence and uniqueness for every $t$ is guaranteed by the hypotheses on $p$ and $m$ and by Lemma \ref{lem:fixpointmeas}. We have
	\begin{align}\begin{split}\label{eq:metricest}
	\lambda_{\omega}(\overline{\mu}_t,\overline{\mu}_s) &\leq \sum^{\infty}_{k = 1} \frac{1}{{(\delta+1)}^{\omega k}}\left|\int_{\overline{\Omega}}(f_k(x,\alpha^*(t,x,p_t(x);\mu_t))m_t(x) - f_k(x,\alpha^*(s,x,p_s(x);\mu_s))m_s(x))dx\right|\\
	&= \sum^{\infty}_{k = 1} \frac{1}{{(\delta+1)}^{\omega k}}\int_{\overline{\Omega}}\left|f_k(x,\alpha^*(t,x,p_t(x);\mu_t)) - f_k(x,\alpha^*(s,x,p_s(x);\mu_s))\right||m_t(x)|dx \\
	&\quad \quad + \sum^{\infty}_{k = 1} \frac{1}{{(\delta+1)}^{\omega k}}\int_{\overline{\Omega}}\left|f_k(x,\alpha^*(t,x,p_s(x);\mu_t))\right|\left|m_t(x) - m_s(x)\right|dx.
	\end{split}
	\end{align}
	
	We now focus on the first term of the right-hand side of \eqref{eq:metricest}. 	Notice that, by Assumption \ref{item:contalpha}-\ref{item:contalpha4} and Lemma \ref{le:techweakconv} with $h = \xi$, $B = \overline{\Omega}\times A$ and our choice of $a$, we have the estimate
	\begin{align*}
	\int_{\overline{\Omega}}|\alpha^*(t,x,p_t(x);\mu_t)) - \alpha^*(t,x,p_t(x);\mu_s)|dx
	& \leq L\Big|\int_{\overline{\Omega}\times A} \xi(x,\alpha)d(\mu_t -\mu_s)(x,\alpha)\Big|\\
	& \leq La(1+2\delta)\|\xi\|_{\infty}d_a(\mu_t,\mu_s)\\
	& \leq La(1+2\delta)\|\xi\|_{\infty}\lambda_\omega(\mu_t,\mu_s),
	\end{align*}
	where $d_a \leq \lambda_\omega$ comes from $a > (\delta+1)^\omega$ (as implied by \eqref{eq:omegachoice}).
	
	Moreover, the family $(f_k)_{k\in\N}$ is a basis of Lipschitz continuous functions with
	$$\Lip(f_k;\overline{\Omega}\times A) \leq \deg(k)\delta^{\deg(k)-1}
	<k{(\delta+1)}^{k}.$$
	Taking these information into account we may write
	\begin{align*}
	\sum^{\infty}_{k = 1} \frac{1}{{(\delta+1)}^{\omega k}}&\int_{\overline{\Omega}}\left|f_k(x,\alpha^*(t,x,p_t(x);\mu_t)) - f_k(x,\alpha^*(s,x,p_s(x);\mu_s))\right||m_t(x)|dx \\
	& \leq \sum^{\infty}_{k = 1} k {(\delta+1)}^{k(1-\omega)}\int_{\overline{\Omega}}\Big|\alpha^*(t,x,p_t(x);\mu_t)) - \alpha^*(s,x,p_s(x);\mu_s)\Big||m_t(x)|dx\\
	&\leq \sum^{\infty}_{k = 1} k {(\delta+1)}^{k(1-\omega)}\int_{\overline{\Omega}}\Big(L_1|t-s|^{\beta/2} + L|p(t,x) - p(s,x)| +\\
	&\quad\quad\quad\quad\quad\quad\quad\quad\quad\quad\quad\quad\quad\quad\quad+ La(1+2\delta)\|\xi\|_{\infty}\lambda_\omega(\mu_t,\mu_s)\Big)|m_t(x)|dx\\
	&\leq \sum^{\infty}_{k = 1} k {(\delta+1)}^{k(1-\omega)}\left(L_1 + L\|p\|_{\Hoelder{\beta}{\beta/2}{\overline{Q}_T;\R^d}}+La(1+2\delta)\|\xi\|_{\infty}\|\mu\|_\beta\right)|t-s|^{\beta/2},
	\end{align*}
	where, as a shorthand notation, we denoted by $\|\mu\|_\beta \defin \|\mu\|_{\mathcal{C}^{\beta/2}([0,T];\mathcal{M}_1(\overline{\Omega}\times A))}$. Notice that we used the H\"older continuity in time of the functions $p$ and $\mu$.
	
	Concerning the second term on the right-hand side of \eqref{eq:metricest}, since it holds that
	$$\sup_{(x,\alpha) \in \overline{\Omega}\times A} |f_k(x,\alpha)| \leq {\delta}^{\deg(k)}< k{(\delta+1)}^{k},$$
	then we also have
	\begin{align*}
	\sum^{\infty}_{k = 1} & \frac{1}{{(\delta+1)}^{\omega k}}\int_{\overline{\Omega}}\left|f_k(x,\alpha^*(t,x,p_s(x);\mu_t))\right|\left|m_t(x) - m_s(x)\right|dx\\
	&\leq\sum^{\infty}_{k = 1} {k(\delta+1)}^{k(1-\omega)}\|m\|_{\Hoelder{\beta}{\beta/2}{\overline{Q}_T;\R}}|t-s|^{\beta/2}.
	\end{align*}
	In conclusion, using $\|m\|_{\Hoelder{\beta}{\beta/2}{\overline{Q}_T;\R}} \leq M_0$, $\|p\|_{\Hoelder{\beta}{\beta/2}{\overline{Q}_T;\R^d}} \leq P_0$ and $\|\mu\|_\beta \leq C_0$, we arrive at
	\begin{align*}
	\lambda_{\omega}(\overline{\mu}_t,\overline{\mu}_s) &\leq \sum^{\infty}_{k = 1}{k(\delta+1)}^{k(1-\omega)}\left(L_1 + LP_0+ M_0 + La(1+2\delta)\|\xi\|_{\infty}C_0\right)|t-s|^{\beta/2}.
	\end{align*}
	The choice of $\omega$ in \eqref{eq:omegachoice} then yields
	\begin{align*}
	\lambda_{\omega}(\overline{\mu}_t,\overline{\mu}_s) \leq \sum^{\infty}_{k = 1}\frac{C_0}{3^k}|t-s|^{\beta/2}= \frac{C_0}{2}|t-s|^{\beta/2}<C_0|t-s|^{\beta/2}.
	\end{align*}
	This, in turn, implies that $\Psi[p,m]$ is an application from the set
	\begin{align*}
	\left\{\mu \in \mathcal{C}^{\beta/2}([0,T];\mathcal{M}_1(\overline{\Omega}\times A)) \mid \|\mu\|_\beta\leq C_0\right\}
	\end{align*}
	onto itself. Notice that the above set is convex and compact in $\mathcal{C}^{\beta/2}([0,T];\mathcal{M}_1(\overline{\Omega}\times A))$ (by the Ascoli-Arzel\'a theorem). Moreover $\Psi[p,m]$ is a continuous application, since the map $\Phi[p_t,m_t]$ is continuous with respect to the weak$^*$ convergence of measures by Lemma \ref{lem:fixpointmeas}. We can therefore invoke the Schauder fixed point theorem to infer the existence of a fixed point for $\Psi[p,m]$. Uniqueness, instead,  follows from uniqueness of the fixed point for $\Phi[p_t,m_t]$ for every $t\in[0,T]$.
\end{proof}

Next, we extend the stability result of Lemma \ref{lem:stability} to the map $\Psi$.

\begin{lemma}\label{lem:hoelderstability}
	Assume that the sequence
	$$(p_n,m_n)_{n\in\N}\subset\Hoelder{\beta}{\beta/2}{\overline{Q}_T;\R^d}\times\Hoelder{\beta}{\beta/2}{\overline{Q}_T;\R} \quad \text{ with } m_n \in \mathcal{M}_1(\overline{\Omega})$$
	converges to $(p,m)_{n\in\N}\in\Hoelder{\beta}{\beta/2}{\overline{Q}_T;\R^d}\times\Hoelder{\beta}{\beta/2}{\overline{Q}_T;\R}$ in the product topology, and consider the family of functions $\vartheta_n(t) = \Gamma(p_{n,t},m_{n,t})$, as well as $\vartheta(t)\defin \Gamma(p_{t},m_{t})$. Then $(\vartheta_n)_{n\in\N}$ converges uniformly to $\vartheta$.
\end{lemma}
\begin{proof}
	Since the sequences $(p_n)_{n\in\N}$ and $(m_n)_{n\in\N}$ are convergent, they are bounded in their respective $\beta$-H\"older norm by two positive constants $P_0$ and $M_0$. Then we can take a uniform $C_0$ and $\omega$ for the entire sequence by using such $P_0$ and $M_0$ in Lemma \ref{lem:hoelderfixpointmeas}. This implies that the metric $\lambda_\omega$ on $\mathcal{M}_1(\overline{\Omega}\times A)$ of the previous result is such that it holds
	\begin{align*}
	(\mu_n)_{n\in\N}\subset Z\defin\left\{\mu \in \mathcal{C}^{\beta/2}([0,T];\mathcal{M}_1(\overline{\Omega}\times A)) \mid \|\mu\|_{\mathcal{C}^{\beta/2}([0,T];\mathcal{M}_1(\overline{\Omega}\times A))}\leq C_0\right\}.
	\end{align*}
	Since the set $Z$ is compact by the Ascoli-Arzel\'a theorem, we can extract from $(\mu_n)_{n\in\N}$ a convergent subsequence to some $\mu \in Z$ and the convergence is uniform in time, that is
	$$\lambda_\omega(\mu_{n,t},\mu_t) \rightarrow 0 \quad \text{ as } \quad n\rightarrow\infty \quad \text{ uniformly in } t.$$
	%
	On the other hand, we know from Lemma \ref{lem:stability} that the sequence $(\Gamma(p_{n,t},m_{n,t}))_{t\in[0,T]}$ converges pointwise to $(\Gamma(p_t,m_t))_{t\in[0,T]} \in Z$, therefore by the uniqueness of the pointwise limit we have
	$$\mu_t = \Gamma(p_t,m_t) \quad \text{ for every } t\in[0,T].$$
	which in turn implies that $(\Gamma(p_{n,t},m_{n,t}))_{t\in[0,T]}$ converges to $(\Gamma(p_t,m_t))_{t\in[0,T]}$ uniformly. 
\end{proof}

Notice that our smoothness assumptions on $b$ and $\mathcal{L}$ transfer easily to the Hamiltonian of the system, as the following technical result shows.

\begin{proposition}\label{prop:holderHamilton}
	The Hamiltonian $H$ satisfies the following statement: let $p \in \Hoelder{\beta}{\beta/2}{\overline{Q}_T;\R^d}$ and let $m \in \Hoelder{\beta}{\beta/2}{\overline{Q}_T;\R}$ with $m \in \mathcal{M}_1(\overline{\Omega})$, then it holds
	$$H(t,x,p(t,x);\Gamma(p_t,m_t)) \in \Hoelder{\beta}{\beta/2}{\overline{Q}_T;\R},$$
	as well as
	$$D_p H(t,x,p(t,x);\Gamma(p_t,m_t)) \in \Hoelder{\beta}{\beta/2}{\overline{Q}_T;\R^d}.$$
\end{proposition}
\begin{proof}
	By Assumption \ref{item:contalpha}, we know that there exists an $\alpha^*:[0,T]\times\R^d\times\R^d\times\mathcal{M}_1(\overline{\Omega}\times A)\rightarrow\R$ for which
	\begin{align}\label{eq:Hrep}
	H(t,x,p;\nu) = -p\cdot b(t,x,\alpha^*(t,x,p;\nu);\nu) - \mathcal{L}(t,x,\alpha^*(t,x,p;\nu);\nu)
	\end{align}
	holds, which in turn implies
	\begin{align}\label{eq:DpHrep}
	D_p H(t,x,p;\nu) = -b(t,x,\alpha^*(t,x,p;\nu);\nu).
	\end{align}
	By Lemma \ref{lem:hoelderfixpointmeas} we have that $\Gamma(p_t,m_t) \in \mathcal{C}^{\beta/2}([0,T];\mathcal{M}_1(\overline{\Omega}\times A))$. The thesis then follows from Assumptions \mattia{\ref{item:inverseb}-\ref{item:inverseb2},  \ref{item:monotoneL}-\ref{item:monotoneL2}} and  \ref{item:contalpha}-\ref{item:contalpha1}.
\end{proof}

We can now show that the stability of the map $\Gamma$ easily translates into stability for the functions $H$ and $D_p H$.

\begin{lemma}\label{le:Hstability}
	Assume that the sequence
	$$(p_n,m_n)_{n\in\N}\subset\Hoelder{\beta}{\beta/2}{[0,T];\R^d}\times \Hoelder{\beta}{\beta/2}{[0,T];\R} \quad \text{ with } m_n \in \mathcal{M}_1(\overline{\Omega})$$
	converges to $(p,m)_{n\in\N}\in\Hoelder{\beta}{\beta/2}{[0,T];\R^d}\times \Hoelder{\beta}{\beta/2}{[0,T];\R}$ in the product topology. Then for $n\rightarrow\infty$ it holds	
	\begin{align*}
	H(t,x,p_{n,t}(x);\Gamma(p_{n,t},m_{n,t})) &\rightarrow H(t,x,p_{t}(x);\Gamma(p_{t},m_{t})) \quad \text{ uniformly,}
	\end{align*} 
	as well as
	\begin{align*}
	D_p H(t,x,p_{n,t}(x);\Gamma(p_{n,t},m_{n,t})) &\rightarrow D_p H(t,x,p_{t}(x);\Gamma(p_{t},m_{t}))  \quad \text{ uniformly.}
	\end{align*} 
\end{lemma}
\begin{proof}
	We shall first show that
	\begin{align*}
	\alpha^*(t,x,p_{n,t}(x);\Gamma(p_{n,t},m_{n,t})) \rightarrow \alpha^*(t,x,p_{t}(x);\Gamma(p_{t},m_{t}))  \quad \text{ uniformly.}
	\end{align*} 
	The thesis will then follow by the regularity assumptions on $b$ and $\mathcal{L}$, namely \ref{item:inverseb}-\ref{item:inverseb2} and \ref{item:monotoneL}-\ref{item:monotoneL2}, and the fact that $H$ satisfies \eqref{eq:Hrep} and \eqref{eq:DpHrep} by Proposition \ref{prop:holderHamilton}.
	
	Let $\omega$ be determined by \eqref{eq:omegachoice} (for $P_0$ and $M_0$ chosen according to the bounds of the sequences $(p_n)_{n \in \N}$ and $(m_n)_{n \in \N}$). Then the following estimate follows easily from Assumption \ref{item:contalpha}-\ref{item:contalpha4} together with Lemma \ref{le:techweakconv}:
	\begin{align*}
	|\alpha^*(t,x,p_{n,t}(x);\Gamma(p_{n,t},m_{n,t}))& - \alpha^*(t,x,p_{t}(x);\Gamma(p_{t},m_{t})) |  \\
	& \leq L|p_{n,t}(x)-p_{t}(x)| \\
	& \quad \quad+ L\left|\int_{\overline{\Omega}\times A}\xi(x,\alpha)d\Gamma(p_{n,t},m_{n,t})(x,\alpha)-\int_{\overline{\Omega}\times A}\xi(x,\alpha)d\Gamma(p_{t},m_{t})(x,\alpha)\right|\\
	& \leq L|p_{n,t}(x)-p_{t}(x)| + L\|\xi\|_{\infty}(\delta+1)^\omega(1+2\delta)\lambda_{\omega}(\Gamma(p_{n,t},m_{n,t}),\Gamma(p_{t},m_{t})).
	\end{align*}
	By taking the supremum over $(t,x)\in \overline{Q}_T$, the statement is proven by invoking Lemma \ref{lem:hoelderstability}.
\end{proof}

The following a priori H\"older bound for $u$ and $m$ follows easily from the regularity assumptions on the data (i.e., the sets, constants and functions listed in the assumptions) of the problem and the previous lemmas.

\begin{lemma}\label{lem:unifboundsol}
	There exists a constant $\overline{C}$ depending only on the data of the problem such that, if $(u,m) \in \Hoelder{2+\beta}{1+\beta/2}{\overline{Q}_T;\R}\times \Hoelder{2+\beta}{1+\beta/2}{\overline{Q}_T;\R}$ with $m_t \in \mathcal{M}_1(\overline{\Omega})$ is a $\beta$-classical solution of \eqref{eq:mfg} then
	\begin{align*}
	\|u\|_{\Hoelder{2+\beta}{1+\beta/2}{\overline{Q}_T;\R}} + \|m\|_{\Hoelder{2+\beta}{1+\beta/2}{\overline{Q}_T;\R}} \leq \overline{C}.
	\end{align*}
\end{lemma}
\begin{proof}
	By Lemma \ref{lem:hoelderfixpointmeas}, if $(u,m) \in \Hoelder{2+\beta}{1+\beta/2}{\overline{Q}_T;\R}\times \Hoelder{2+\beta}{1+\beta/2}{\overline{Q}_T;\R}$ then the measure valued curve $\mu_t = \Gamma(Du(t,\cdot),m_t)$ for every $t\in[0,T]$ is H\"older continuous in time with exponent $\beta/2$. Moreover, by Proposition \ref{prop:holderHamilton} follows that the coefficients of the PDEs \eqref{eq:mfghjb} and \eqref{eq:fokkerplanck} are bounded in the parabolic H\"older space $\Hoelder{\beta}{\beta/2}{\overline{Q}_T;\R}$ and $\Hoelder{\beta}{\beta/2}{\overline{Q}_T;\R^d}$, respectively.
	
	Therefore, since $u\in\Hoelder{2+\beta}{1+\beta/2}{\overline{Q}_T;\R}$, $m$ is a solution to a nondegenerate parabolic equation with coefficients bounded in $\Hoelder{\beta}{\beta/2}{\overline{Q}_T;\R}$. We can thus invoke \cite[Theorem V.6.1, Equation (6.12)-(6.12')]{solonnikov} to get the uniform bound
	\begin{align}\label{eq:unifboundm}
	\|m\|_{\Hoelder{2+\beta}{1+\beta/2}{\overline{Q}_T;\R}}  \leq C_1,
	\end{align}
	for some constant $C_1$ depending only on the data. For the same reasons, since $m\in\Hoelder{2+\beta}{1+\beta/2}{\overline{Q}_T;\R}$, 
	then $u$ also solves a nondegenerate parabolic equation with coefficients in $\Hoelder{\beta}{\beta/2}{\overline{Q}_T;\R}$ and it satisfies
	\begin{align}\label{eq:unifboundu}
	\|u\|_{\Hoelder{2+\beta}{1+\beta/2}{\overline{Q}_T;\R}}  \leq C_2,
	\end{align}
	where the constant $C_2$ depends only on the data. Putting together the two estimates we get the thesis for $\overline{C}\defin C_1+C_2$.
\end{proof}

\subsection{Well-posedness of system \eqref{eq:mfg}}

We are now ready to prove the main result of the paper.

\begin{proof}[Proof of Theorem \ref{th:mainresult}]
	Denote by $X$ the subset of $\Hoelder{2+\beta}{1+\beta/2}{\overline{Q}_T;\R}\times\Hoelder{2+\beta}{1+\beta/2}{\overline{Q}_T;\R}$ such that if $(u,m)\in X$ then $m_t\in\mathcal{M}_1(\overline{\Omega})$ for every $t\in[0,T]$, and equip it with the product norm (which we denote by $\|\cdot\|_X$).
	\medskip
	
	\begin{claim}
		$X$ is a closed subset of $\Hoelder{2+\beta }{1+\beta/2}{\overline{Q}_T;\R}\times\Hoelder{2+\beta }{1+\beta/2}{\overline{Q}_T;\R}$.
	\end{claim}
	\begin{claimproof}
		We just need to check that if $(u_n,m_n)_{n\in\N}$ is a sequence in $X$ converging to $(u,m)$, then $m_t\in\mathcal{M}_1(\overline{\Omega})$ for every $t\in[0,T]$. However, since $(m_n)_{n\in\N}$ converges to $m$ uniformly then it holds
		$$\lim_{n\rightarrow\infty}\int_{\overline{\Omega}}m_n(x)dx = \int_{\overline{\Omega}}m(x)dx.$$
		Therefore, by the fact that $m_n(\overline{\Omega}) \leq 1$ for every $n\in\N$, we have the desired statement.
	\end{claimproof}
	\medskip
	
	We now introduce the map $T:X\times [0,1]\rightarrow X$ that associates to every $(u,m) \in X$ and every $\eps \in\ [0,1]$ a solution $(v,\rho)$ of the PDE system
	\begin{align}\label{eq:mfgfixedpoint}
	\left\{\begin{aligned}
	&\partial_t v(t,x) + \sigma \Delta v(t,x) - \eps H(t,x,D u(t,x);\mu_t) = 0, & \text{ for } x\in\Omega \text{ and } t \in[0,T]\\
	&\partial_t \rho(t,x) - \sigma \Delta \rho(t,x) - \eps\diver(D_p H(t,x,D u(t,x);\mu_t)\rho(t,x)) = 0, & \text{ for } x\in\Omega \text{ and } t \in[0,T]\\
	&\mu_t = (\Id,\alpha^*(t,\cdot,D u(t,\cdot);\mu_t))_{\sharp}m_t&\text{ for } t \in[0,T]\\
	&\rho(0,x) = \eps\overline{m}(x), v(T,x) = \eps\psi(x,m_T) & \text{ for } x\in\overline{\Omega} \\
	&\rho(t,x) = 0, v(t,x) = \eps\psi(x,m_t) & \text{ for } x\in\partial\Omega\text{ and } t \in[0,T]
	\end{aligned}\right.
	\end{align}
	
	\medskip
	
	\begin{claim}{$T(\cdot,\eps)$ is a well-defined operator for every $\eps \in [0,1]$.}
	\end{claim}
	\begin{claimproof}
		Set $(v,\rho) \defin T(u,m,\varepsilon)$. Since the couple $(Du,m)$ satisfies the assumptions of Lemma \ref{lem:hoelderfixpointmeas}, the measure-valued curve $\mu_t = \Gamma(Du_t,m_t)$ satisfies the fixed-point relationship
	$$\mu_t = (\Id,\alpha^*(t,\cdot,D u(t,\cdot);\mu_t))_{\sharp}m_t,$$
	for every $t\in[0,T]$. Morever, $\mu_t$ belongs to $\mathcal{C}^{\beta/2}([0,T];\mathcal{M}_1(\overline{\Omega}\times A))$.
	
	Now, by Proposition \ref{prop:holderHamilton}, if $(u,m) \in X$ then the term $H(t,x,D u(t,x);\mu_t)$ satisfies the hypotheses of \cite[Theorem IV.5.2]{solonnikov}, hence there exists a unique solution $v\in \Hoelder{2+\beta }{1+\beta/2}{\overline{Q}_T}$ of
	\begin{align}\label{eq:hjbfixedpoint}
	\left\{\begin{aligned}
	&\partial_t v(t,x) + \sigma \Delta v(t,x) - \eps H(t,x,D u(t,x);\mu_t) = 0, & \text{ for } x\in\Omega \text{ and } t \in[0,T]\\
	&v(T,x) = \eps \psi(x,m_T) & \text{ for } x\in\overline{\Omega} \\
	&v(t,x) = \eps \psi(x,m_t) & \text{ for } x\in\partial\Omega
	\end{aligned}\right.
	\end{align}
	such that
	\begin{align}\begin{split}\label{eq:boundv}
	\|v\|_{\Hoelder{2+\beta}{1+\beta/2}{\overline{Q}_T;\R}} &\leq C_1 \left(\|H(t,x,D u(t,x);\mu_t)\|_{\Hoelder{\beta}{\beta/2}{\overline{Q}_T;\R}} + \|\psi(x,m_t)\|_{\Hoelder{\beta}{\beta/2}{\overline{Q}_T;\R}}\right)\\
	&\leq C_2 \left(\mathcal{H}_1\left(\|D u(t,x)\|_{\Hoelder{\beta}{\beta/2}{\overline{Q}_T;\R^d}} + \|m\|_{\Hoelder{\beta}{\beta/2}{\overline{Q}_T;\R}}\right) + \mathcal{K}_{\psi}\left(\|m\|_{\Hoelder{\beta}{\beta/2}{\overline{Q}_T;\R}}\right)\right)\\
	&\leq \mathcal{K}_1\left(\|(u,m)\|_X\right),
	\end{split}
	\end{align}
	for some continuous function $\mathcal{K}_1:\R\rightarrow\R$ and constants $C_1,C_2>0$ depending only on the data. Here we used Assumption \ref{item:hoelderpsi}-\ref{item:hoelderpsi1}, as well as Assumption \ref{item:hoelderH}.
	
	Furthermore, since $D_p H(t,x,D u(t,x);\mu_t)$ is H\"older continuous as well, for the same reasons there exists also a solution $\rho \in \Hoelder{2+\beta }{1+\beta/2}{\overline{Q}_T}$ of
	\begin{align}\label{eq:fpfixedpoint}
	\left\{\begin{aligned}
	&\partial_t \rho(t,x) - \sigma \Delta \rho(t,x) - \eps\diver(D_p H(t,x,D u(t,x);\mu_t)\rho(t,x)) = 0, & \text{ for } x\in\Omega \text{ and } t \in[0,T]\\
	&\rho(0,x) = \eps\overline{m}_0(x) & \text{ for } x\in\overline{\Omega} \\
	&\rho(t,x) = 0 & \text{ for } x\in\partial\Omega
	\end{aligned}\right.
	\end{align}
	which satisfies
	\begin{align}\begin{split}
	\label{eq:boundrho}
	\|\rho\|_{\Hoelder{2+\beta }{1+\beta/2}{\overline{Q}_T;\R}} &\leq C_3 \left(\|D_p H(t,x,D u(t,x);\mu_t)\|_{\Hoelder{\beta}{\beta/2}{\overline{Q}_T;\R^d}} + \|\overline{m}_0\|_{\mathcal{C}^{\beta}(\overline{\Omega};\R)}\right)\\
	&\leq C_4 \left(\mathcal{H}_2\left(\|D u(t,x)\|_{\Hoelder{\beta}{\beta/2}{\overline{Q}_T;\R^d}} + \|m\|_{\Hoelder{\beta}{\beta/2}{\overline{Q}_T;\R}}\right) + \|m\|_{\Hoelder{\beta}{\beta/2}{\overline{Q}_T;\R}}\right)\\
	&\leq \mathcal{K}_2\left(\|(u,m)\|_X\right),
	\end{split}
	\end{align}
	again for some continuous function $\mathcal{K}_2:\R\rightarrow\R$ and constants $C_3,C_4>0$ depending only on the data, having used Assumption \ref{item:hoelderH}.
	
	Moreover, by Proposition \ref{prop:measfp}, it holds that $\rho_t\in\mathcal{M}_1(\overline{\Omega})$ for every $t\in[0,T]$.
	
	Summing \eqref{eq:boundv} and \eqref{eq:boundrho} together we arrive at
	\begin{align}\begin{split}\label{eq:contboundT}
	\|T(u,m,\eps)&\|_{\Hoelder{2+\beta }{1+\beta/2}{\overline{Q}_T;\R}\times\Hoelder{2+\beta }{1+\beta/2}{\overline{Q}_T;\R}} \\
	& = \|(v,\rho)\|_{\Hoelder{2+\beta }{1+\beta/2}{\overline{Q}_T;\R}\times\Hoelder{2+\beta }{1+\beta/2}{\overline{Q}_T;\R}}\\
	& = \|v\|_{\Hoelder{2+\beta }{1+\beta/2}{\overline{Q}_T;\R}} + \|\rho\|_{\Hoelder{2+\beta }{1+\beta/2}{\overline{Q}_T;\R}}\\
	& \leq \mathcal{K}\left(\|(u,m)\|_X\right),
	\end{split}
	\end{align}
	for some continuous function $\mathcal{K}:\R\rightarrow\R$ depending only on the data. It then follows that $(v,\rho) = T(u,m,\eps) \in X$, and hence $T(\cdot,\eps)$ is a well-defined map.
\end{claimproof}

\medskip

	\begin{claim}
		$T(\cdot,\eps)$ is a continuous map for every $\eps \in [0,1]$.
	\end{claim}
	\begin{claimproof}
	Let $(u_n,m_n)_{n\in\N}\in X$ be a sequence converging to $(u,m)$ in the norm of $X$, and denote by $(v_n,\rho_n) \defin T(u_n,m_n,\eps)$ for every $n\in\N$. Then $(u_n,m_n)_{n\in\N}\in X$ is also bounded, and so is $\mathcal{K}(\|(u_n,m_n)\|_X)_{n\in\N}$ by the continuity of the function $\mathcal{K}$ in \eqref{eq:contboundT}. Therefore, by the estimate \eqref{eq:contboundT} we infer that $(v_n,\rho_n)_{n\in\N}\in X$ is bounded in $\Hoelder{2+\beta }{1+\beta/2}{\overline{Q}_T;\R}\times\Hoelder{2+\beta }{1+\beta/2}{\overline{Q}_T;\R}$ as well. Hence, by Ascoli-Arzel\'a Theorem and the fact that $X$ is closed, we may infer the existence of $(v,\rho)\in X$ such that $(v_n,\rho_n)_{n\in\N}$ converges to $(v,\rho)$ in H\"older norm (up to subsequences). We now have to show that $T(u,m,\eps) = (v,\rho)$.
	
	 To do so notice that, since $(D u_n)_{n\in\N}$ is converging in $\Hoelder{\beta}{\beta/2}{\overline{Q}_T;A}$ and $(m_n)_{n\in\N}$ in $\Hoelder{\beta}{\beta/2}{\overline{Q}_T;\R}$ with $m_n \in \mathcal{M}_1(\overline{\Omega})$, we can invoke Lemma \ref{lem:hoelderstability} to infer the convergence (up to subsequences) in $\mathcal{C}^{\beta/2}$ of the sequence $(\mu_n)_{n\in\N} \subset \mathcal{C}^{\beta/2}(\overline{Q}_T;\mathcal{M}_1(\overline{\Omega}\times A))$ defined as
	$$\mu_n \defin (\Gamma(D u_{n}(t,\cdot),m_{n,t}))_{t\in[0,T]}$$
	to $\mu = (\Gamma(D u(t,\cdot),m_{t}))_{t\in[0,T]}$. Since gradients converge uniformly, by Lemma \ref{le:Hstability} it follows that
	\begin{align*}
	H(t,x,D u_n(t,x);\mu_{n,t}) &\rightarrow H(t,x,D u(t,x);\mu_{t}) \quad \text{ uniformly.}
	\end{align*} 
	Therefore, using well-known stability results for viscosity solutions of second-order equations, it holds that $v$ is a classical solution of \eqref{eq:hjbfixedpoint}. Similarly, since by Lemma \ref{le:Hstability} we also get
	\begin{align*}
	D_p H(t,x,D u_n(t,x);\mu_{n,t}) &\rightarrow D_p H(t,x,D u_n(t,x);\mu_{n,t})  \quad \text{ uniformly,}
	\end{align*} 
	then we have that $\rho$ solves \eqref{eq:fpfixedpoint}, which in turn implies $T(u,m,\eps) = (v,\rho)$.
	\end{claimproof}

\medskip

	\begin{claim}
		$T(\cdot,\eps)$ is a compact operator for every $\eps \in [0,1]$.
	\end{claim}
	\begin{claimproof}
		To show compactness we need to prove that it maps bounded sets into precompact sets. Therefore, fix $M>0$ and consider the bounded set
		$$X_M\defin \{(u,m)\in X\mid \|(u,m)\|_X\leq M \}.$$
		By \eqref{eq:contboundT} it holds that $T(\cdot,\eps)$ maps $X_M$ into the subset of $\Hoelder{2+\beta }{1+\beta/2}{\overline{Q}_T;\R}\times\Hoelder{2+\beta }{1+\beta/2}{\overline{Q}_T;\R}$ with H\"older norm bounded by the constant $\sup_{[0,M]}\mathcal{K}$. Since $X$ is closed by a previous claim, the image of $X_M$ by $T(\cdot,\eps)$ is precompact by the Ascoli-Arzel\'a theorem, as desired.
	\end{claimproof}

\medskip
		
	\begin{claim}
	$T(\cdot,\eps)$ has a fixed point for every $\eps \in [0,1]$.
	\end{claim}
	\begin{claimproof}
		We shall apply the Leray-Schauder fixed point theorem. We already showed that for every $\eps\in[0,1]$ the map $T(\cdot,\eps)$ is a continuous compact mapping, and it trivially holds $T(u,m,0) = (0,0)$ by the maximum principle for the heat equation. Therefore, assume that $(u,m)$ is a fixed point of $T(\cdot,\eps)$, i.e., it is a classical solution of \eqref{eq:mfg} with $\eps H$, $\eps D_p H$, $\eps \psi$ and $\eps \overline{m}$ in place of $H$, $D_pH$, $\psi$ and $\overline{m}$. Then from \eqref{eq:unifboundm} and \eqref{eq:unifboundu} in the proof of Lemma \ref{lem:unifboundsol} we see that we can select a constant $\overline{C}>0$ independent from $\eps \in [0,1]$ such that
		$$\|(u,m)\|_{\Hoelder{2+\beta }{1+\beta/2}{\overline{Q}_T;\R}\times\Hoelder{2+\beta }{1+\beta/2}{\overline{Q}_T;\R}} \leq \overline{C}.$$
		Therefore, by the Leray-Schauder fixed point theorem we have that the application $T(\cdot,\eps)$ has a fixed point for every $\eps\in[0,1]$.
	\end{claimproof}

	To conclude the proof, we simply notice that a fixed point $(u,m)$ of $T(\cdot,1)$ is also a classical solution of \eqref{eq:mfg}.
\end{proof}

\section{Modeling applications}\label{sec:examples}

We now present a set of sufficient conditions for some of the assumptions reported in Section \ref{sec:assumption}, and then we shall provide two models that satisfy them, so that we may infer for them the existence of Nash equilibria by Theorem \ref{th:mainresult}.

\subsection{Sufficient conditions for Assumption \ref{item:monotoneL}}

We shall first focus on the Lagrangian and we will show a functional form satisfying Assumptions \ref{item:monotoneL} which can be easily found in the context of multi-agent models. In such contexts, one of the most employed modeling tools is the convolution of the mass of agents with an \textit{interaction kernel}, which models the interaction of the mass with itself. We recall that the \textit{convolution} between two functions $f:Y_1\rightarrow Y_2$ and $m \in \mathcal{M}_1(Y_1)$ is the function $f* m:Y_1\rightarrow Y_2$ defined as
$$f* m(x) \defin \int_{Y_1} f(x - y) dm(y),$$
whenever this quantity is well-defined (for instance, when $m$ is absolutely continuous with respect to the Lebesgue measure on $\R^d$ and $f \in L_1(\R^d)$). Convolutions with sufficiently regular interaction kernels play a crucial role, both in the dynamics as well as inside the cost functional, in several pedestrian and financial markets models (see for instance \cite{carrillo1501.07054} for the former and \cite{ahn13} for the latter).

\begin{definition}[Multi-agent interaction Lagrangian]
	Fix $\ell:[0,T]\times\overline{\Omega}\rightarrow\R$, $h:\R\rightarrow\R$, $Q:\R^d\rightarrow\R$, $g:A\times\mathcal{M}_1(\overline{\Omega}\times A)\rightarrow\R$ and $\varepsilon \geq 0$. We define the \textit{multi-agent interaction Lagrangian} to be
	\begin{align}\label{eq:multiagentL}
	\mathcal{L}(t,x,\alpha;\nu) \defin \ell(t,x) + h((Q*\pi_{1\sharp}\nu)(x)) + g(\alpha,\nu) + \varepsilon|\alpha|^2.
	\end{align}
\end{definition}

For this kind of cost function, a rather easy sufficient condition can be formulated for Assumption \ref{item:monotoneL}-\ref{item:monotoneL1}.

\begin{proposition}\label{prop:monotoneLprop}
	If there exists $F:\overline{\Omega}\times A\rightarrow\R$ such that for every $(t,x,\alpha)\in[0,T]\times\overline{\Omega}\times A$ and $\nu^1,\nu^2 \in\mathcal{M}_1(\overline{\Omega}\times A)$ it holds
	\begin{align}\begin{split}\label{eq:suffmonotone}
	h((Q*\pi_{1\sharp}\nu^1)(x)) +g(\alpha,\nu^1) - h((Q*\pi_{1\sharp}\nu^2)(x))& -g(\alpha,\nu^2) \\
	& \geq F(x,\alpha)\int_{\overline{\Omega}\times A}F(\tilde{x},\tilde{\alpha})d(\nu^1-\nu^2)(\tilde{x},\tilde{\alpha}),
	\end{split}
	\end{align}
	then the Lagrangian \eqref{eq:multiagentL} satisfies Assumption \ref{item:monotoneL}-\ref{item:monotoneL1}.
\end{proposition}
\begin{proof}
	Follows easily from the fact that
	\begin{align*}
	\int_{\overline{\Omega}\times A}& (\mathcal{L}(t,x,\alpha;\nu^1) - \mathcal{L}(t,x,\alpha;\nu^2))d(\nu^1 - \nu^2)(x,\alpha)\\
	& = \int_{\overline{\Omega}\times A}\left(h((Q*\pi_{1\sharp}\nu^1)(x)) +g(\alpha,\nu^1) - h((Q*\pi_{1\sharp}\nu^2)(x)) -g(\alpha,\nu^2)\right)d(\nu^1 - \nu^2)(x,\alpha) \\
	& \geq \left[\int_{\overline{\Omega}\times A}F(x,\alpha)d(\nu^1-\nu^2)(x,\alpha)\right]^2 \geq 0.
	\end{align*}
	This concludes the proof.
\end{proof}

The regularity condition of Assumption \ref{item:monotoneL}-\ref{item:monotoneL2} requires to be more precise in the choice of the function $g$. As we want to model situations where agents optimize their choices taking into account the average strategy of the other agents, we first define the operator
\begin{align}\label{eq:intmu}
\Theta(\nu) \defin \int_{\overline{\Omega}\times A}\alpha d\nu(x,\alpha).
\end{align}
The operator $\Theta$ measures the average control strategy of the mass of players. Such quantity will be considered inside the Lagrangian in the form of the following cost
\begin{align}\label{eq:gintmu}
g(\alpha,\nu) \defin \varphi(\alpha \cdot \Theta(\nu)),
\end{align}
where $\varphi$ is a sufficiently smooth, convex and monotone function. Minimizing such cost leads the infinitesimal agents to align its strategy $\alpha^*$ with the direction of the average strategy $\Theta(\nu)$, provided that $\varphi$ is decreasing. On the other hand, if $\varphi$ is increasing, the optimal strategy $\alpha^*$ should go on the opposite direction of $\Theta(\nu)$.

\begin{remark}
	In several applications, the function $g$ \mattia{may be} chosen as
	\begin{align}\label{eq:gintmuclassic}
	g(\alpha,\nu) \defin \varphi(\alpha - \Theta(\nu)),
	\end{align}
	for some smooth, convex and monotone function $\varphi$ (with $\varphi(z) = z^2$ being the favourite choice, see e.g. \cite{MR4146720,kobeissi2019classical}). Our unconventional choice of $g$ is done on the basis that all the results that will be proven for \eqref{eq:gintmu} can be trivially extended to \eqref{eq:gintmuclassic} as well, and that our cost leads to nice closed loop solutions (see, for instance, Remark \ref{rem:expcontrol}) and to interesting agent dynamics, as shown in Section \ref{sec:numerics}.
\end{remark}

We first show that the H\"older regularity in time of a measure $\nu$ passes naturally to $\Theta(\nu)$, as the following result shows.
\begin{proposition}\label{prop:holderintegral}
	Assume that $\mu\in\mathcal{C}^{\beta/2}([0,T];\mathcal{M}_1(\overline{\Omega}\times A))$ with respect to the metric $\lambda_{\omega}$ of Lemma \ref{lem:hoelderfixpointmeas}. Then the function $\Theta(\mu_t)$ 
	belongs to $\mathcal{C}^{\beta/2}([0,T];\R)$.
\end{proposition}
\begin{proof}
	By Lemma \ref{le:techweakconv} for the choice $h = \alpha$, $a = (\delta+1)^\omega$ and $B = \overline{\Omega}\times A$, for every $t,s\in[0,T]$ we have
	\begin{align*}
	|\Theta(\mu_t) - \Theta(\mu_s)| & \leq \left|\int_{\overline{\Omega}\times A}\alpha d\mu_t(x,\alpha) - \int_{\overline{\Omega}\times A}\alpha d\mu_s(x,\alpha)\right| \\
	& \leq \delta(\delta+1)^{\omega}(1+2\delta)\lambda_\omega(\mu_t,\mu_s)\\
	& \leq \delta(\delta+1)^{\omega}(1+2\delta) |\mu|_{\mathcal{C}^{\beta}([0,T])}|t-s|^{\beta/2}.
	\end{align*}
	This concludes the proof.
\end{proof}

This result allows us to extend the regularity of $\alpha$ and $\nu$ to the multi-agent interaction Langrangian, as the following result shows.

\begin{proposition}\label{prop:A5prop}
	Assume that $\ell$, $h$, $Q$ and $\varphi$ are smooth functions and that $g$ satisfies \eqref{eq:gintmu}. Then the multi-agent interaction Lagrangian \eqref{eq:multiagentL} satisfies Assumption \ref{item:monotoneL}-\ref{item:monotoneL2}.
\end{proposition}
\begin{proof}
	Assume that $\alpha \in\Hoelder{\beta}{\beta/2}{\overline{Q}_T;A}$ and that $\nu_t \in \mathcal{C}^{\beta/2}([0,T];\mathcal{M}_1(\overline{\Omega}\times A))$. By Proposition \ref{prop:holderintegral} we have that $\Theta(\nu_t)$ belongs to $\mathcal{C}^{\beta/2}([0,T];\R)$, so that Lemma \ref{le:lipbasic} implies that $g(\alpha; \nu_t)$ is.
	
	Concerning the convolution term, notice that for every $(x_1,t_1), (x_2,t_2) \in \overline{Q}_T$ it holds
	\begin{align*}
	|h((Q*\pi_{1\sharp}&\nu_{t_1})(x_1)) - h((Q*\pi_{1\sharp}\nu_{t_2})(x_2))| \\ 
	& \leq \Lip(h)|(Q*\pi_{1\sharp}\nu_{t_1})(x_1) - (Q*\pi_{1\sharp}\nu_{t_2})(x_2)|\\
	& = \Lip(h)\left|(Q*\pi_{1\sharp}\nu_{t_1})(x_1) - (Q*\pi_{1\sharp}\nu_{t_1})(x_2) + (Q*\pi_{1\sharp}\nu_{t_1})(x_2) - (Q*\pi_{1\sharp}\nu_{t_2})(x_2)\right|\\
	& \leq \Lip(h)\int_{\overline{\Omega}}\left|Q(x_1 - y) - Q(x_2 - y)\right||\pi_{1\sharp}\nu_{t_1}(y)|dy \\
	& \qquad + \Lip(h)\int_{\overline{\Omega}}\left|Q(x_2 - y)\right||\pi_{1\sharp}\nu_{t_1}(y) - \pi_{1\sharp}\nu_{t_2}(y)|dy,
	\end{align*}
	hence the desired regularity comes from the smoothness of $h$, $Q$ and the $\mathcal{C}^{\beta/2}$ regularity in time for $\nu_t$ (which passes to $\pi_{1\sharp}\nu_t$).
	
	Since the sum of H\"older continuous functions is still H\"older continuous, we get $\mathcal{L}(t,x,\alpha(t,x);\nu_t) \in \Hoelder{\beta}{\beta/2}{\overline{Q}_T;\R}$, which is the desired statement.
\end{proof}

\subsection{Sufficient conditions for Assumption \ref{item:contalpha}}

We now pass to give sufficient conditions for Assumption \ref{item:contalpha}. The framework we consider here is coherent with the one in the previous subsection, so that we will be satisfying the sufficient conditions for Assumption \ref{item:monotoneL} at basically no cost if we already satisfying the following sufficient conditions.

\begin{proposition}\label{prop:A6prop}
	Fix $M > 0$ and $\varepsilon > 0$ such that $0 < 2\varepsilon - 6M^3 - 9M^2 - 3M \leq 2\varepsilon \leq 1$. Assume that
	\begin{enumerate}
		\item $\delta(\overline{\Omega}\times A)\leq M$ and $A \subseteq B_M(0)$,
		\item $b(t,x,\alpha;\nu) = b_1(t,x,\nu) + \alpha$ for some $b_1:[0,T]\times\overline{\Omega}\times\mathcal{M}_1(\overline{\Omega}\times A)\rightarrow \R^d$
		\item $\mathcal{L}(t,x,\alpha;\nu) = \ell(t,x,\nu) + \varphi\left(\alpha\cdot\Theta(\nu)\right) + \varepsilon|\alpha|^2$ where $\ell:[0,T]\times\overline{\Omega}\times\mathcal{M}_1(\overline{\Omega}\times A)\rightarrow \R^d$ and $\varphi:\R\rightarrow\R$ satisfies
		\begin{enumerate}
			\item $\varphi$ is convex,
			\item\label{item:A6supprime} $\Lip(\varphi^\prime) < \frac{2\varepsilon - 6M^3 - 9M^2 - 3M}{M^{2}}$,
			\item\label{item:A6lipprime} $\|\varphi^\prime\|_{\infty} \leq 2\varepsilon - \Lip(\varphi^\prime)M^2$.
		\end{enumerate}
	\end{enumerate}
	Then there exists a unique function $\alpha^*:[0,T]\times\overline{\Omega}\times\R^d\times\mathcal{M}_1(\overline{\Omega}\times A)\rightarrow A$ satisfying \eqref{eq:alphamaximum} and Assumption \ref{item:contalpha}-\ref{item:contalpha4} for the choice $\xi(x,\alpha) \defin \alpha$. 
\end{proposition}
\begin{proof}
	If $\varphi$ is convex, with the above choice of $b$ and $\mathcal{L}$, the Lagrangian function $\mathcal{J}$ is strictly convex in $\alpha$, and thus admits a unique minimizer $\alpha^*$ for every $(t,x,p,\nu) \in [0,T]\times\overline{\Omega}\times\R^d\times\mathcal{M}_1(\overline{\Omega}\times A)$. Since, under the above hypotheses, the Hamiltonian $H$ reads
	\begin{align*}
	H(t,x,p,\nu) = \sup_{\alpha}\left\{-p\cdot b_1(t,x,\nu) - p\cdot \alpha - \ell(t,x,\nu) - \varphi\left(\alpha\cdot\Theta(\nu)\right) - \varepsilon|\alpha|^2\right\},
	\end{align*}
	then optimal control $\alpha^*(t,x,p,\nu)$ satisfies the identity
	\begin{align}\label{eq:firstoder}
	p + 2\varepsilon\alpha + \varphi^\prime\left(\alpha\cdot\Theta(\nu)\right)\Theta(\nu) = 0.
	\end{align}
	
	In order to prove that Assumption \ref{item:contalpha}-\ref{item:contalpha4} holds, let $p^1, p^2 \in \R^d$ and $\nu^1, \nu^2 \in \mathcal{M}_1(\overline{\Omega}\times A)$ and denote by $\alpha_i \defin \alpha^*(t,x,p_i,\nu_i)$ for every $i = 1,2$. Then by \eqref{eq:firstoder} follows
	\begin{align*}
	2\varepsilon|\alpha_1 - \alpha_2| & = |p_1 + \varphi^\prime\left(\alpha_1\cdot\Theta(\nu_1)\right)\Theta(\nu_1) -p_2 - \varphi^\prime\left(\alpha_2\cdot\Theta(\nu_2)\right)\Theta(\nu_2)| \\
	& \leq |p_1 - p_2| + |\varphi^\prime\left(\alpha_1\cdot\Theta(\nu_1)\right)\Theta(\nu_1) - \varphi^\prime\left(\alpha_2\cdot\Theta(\nu_2)\right)\Theta(\nu_2)| \\
	& \leq |p_1 - p_2| + |\varphi^\prime\left(\alpha_1\cdot\Theta(\nu_1)\right)\Theta(\nu_1) - \varphi^\prime\left(\alpha_1\cdot\Theta(\nu_1)\right)\Theta(\nu_2)| \\
	& \qquad + |\varphi^\prime\left(\alpha_1\cdot\Theta(\nu_1)\right)\Theta(\nu_2) - \varphi^\prime\left(\alpha_2\cdot\Theta(\nu_2)\right)\Theta(\nu_2)|\\
	& \leq |p_1 - p_2| + |\varphi^\prime\left(\alpha_1\cdot\Theta(\nu_1)\right)||\Theta(\nu_1) - \Theta(\nu_2)| \\
	& \qquad + |\Theta(\nu_2)||\varphi^\prime\left(\alpha_1\cdot\Theta(\nu_1)\right) - \varphi^\prime\left(\alpha_2\cdot\Theta(\nu_2)\right)|\\
	& \leq |p_1 - p_2| + \|\varphi^\prime\|_{\infty}|\Theta(\nu_1) - \Theta(\nu_2)| \\
	& \qquad + \Lip(\varphi^\prime)|\Theta(\nu_2)||\alpha_1\cdot\Theta(\nu_1)-\alpha_2\cdot\Theta(\nu_2)|\\
	& \leq |p_1 - p_2| + \|\varphi^\prime\|_{\infty}|\Theta(\nu_1) - \Theta(\nu_2)| \\
	& \qquad + \Lip(\varphi^\prime)|\Theta(\nu_2)||\Theta(\nu_1)||\alpha_1-\alpha_2| + \Lip(\varphi^\prime)|\Theta(\nu_2)||\alpha_2||\Theta(\nu_1) - \Theta(\nu_2)|\\
	& = |p_1 - p_2| + \left(\|\varphi^\prime\|_{\infty} +\Lip(\varphi^\prime)|\Theta(\nu_2)||\alpha_2| \right)|\Theta(\nu_1) - \Theta(\nu_2)| \\
	& \qquad + \Lip(\varphi^\prime)|\Theta(\nu_2)||\Theta(\nu_1)||\alpha_1-\alpha_2|.
	\end{align*}
	Since $A \subseteq B_M(0)$ then it follows $|\alpha_2| \leq M$ as well as
	$$|\Theta(\nu)| = \left|\int_{\overline{\Omega}\times A}\alpha d\nu(x,\alpha)\right| \leq \int_{\overline{\Omega}\times A}\left|\alpha\right| d\nu(x,\alpha) \leq M\nu(\overline{\Omega}\times A) \leq M.$$
	We can therefore conclude the above chain of inequalities with
	\begin{align*}
	2\varepsilon|\alpha_1 - \alpha_2| & \leq |p_1 - p_2| + \left(\|\varphi^\prime\|_{\infty} +\Lip(\varphi^\prime)M^2 \right)|\Theta(\nu_1) - \Theta(\nu_2)| \\
	& \qquad + \Lip(\varphi^\prime)M^2|\alpha_1-\alpha_2|,
	\end{align*}
	which, after rearranging, yields
	\begin{align*}
	\left(2\varepsilon - \Lip(\varphi^\prime)M^2\right)|\alpha_1 - \alpha_2| & \leq |p_1 - p_2| + \left(\|\varphi^\prime\|_{\infty} +\Lip(\varphi^\prime)M^2 \right)\left|\int_{\overline{\Omega}\times A}\alpha d(\nu_1-\nu_2)(x,\alpha)\right|,
	\end{align*}
	and since $\Lip(\varphi^\prime) < 2\varepsilon M^{-2}$ by assumption \eqref{item:A6lipprime} and the choice of $M$ and $\varepsilon$, we have
	\begin{align*}
	|\alpha_1 - \alpha_2| & \leq \frac{1}{2\varepsilon - \Lip(\varphi^\prime)M^2}|p_1 - p_2| + \frac{\|\varphi^\prime\|_{\infty} +\Lip(\varphi^\prime)M^2 }{2\varepsilon - \Lip(\varphi^\prime)M^2}\left|\int_{\overline{\Omega}\times A}\alpha d(\nu_1-\nu_2)(x,\alpha)\right|\\
	&\leq \frac{1}{2\varepsilon - \Lip(\varphi^\prime)M^2}\left(|p_1 - p_2| + \left|\int_{\overline{\Omega}\times A}\alpha d(\nu_1-\nu_2)(x,\alpha)\right|\right),
	\end{align*}
	where we used the fact that, by assumption \eqref{item:A6supprime}, it holds $\|\varphi^\prime\|_{\infty} \leq 2\varepsilon- \Lip(\varphi^\prime)M^2$, whence $\|\varphi^\prime\|_{\infty} + \Lip(\varphi^\prime)M^2 \leq 2\varepsilon \leq 1$. It remains to show that the Lipschitz coefficient satisfies condition \eqref{eq:alphalipbound}, that is
	\begin{align*}
	\frac{1}{2\varepsilon - \Lip(\varphi^\prime)M^2} < \frac{1}{(3+3\delta)(1+2\delta)\|\xi\|_{\infty}},
	\end{align*}
	which is equivalent to
	\begin{align*}
	\Lip(\varphi^\prime)M^2 < 2\varepsilon - (3+3\delta)(1+2\delta)\|\xi\|_{\infty}.
	\end{align*}
	However, since by assumption $\delta \leq M$ and, due to the choice of $\xi(x,\alpha) = \alpha$, we have that $\|\xi\|_{\infty} \leq M$, then it is enough to show that
	\begin{align*}
	\Lip(\varphi^\prime)M^2 < 2\varepsilon - (3+3M)(1+2M)M,
	\end{align*}
	which is, however, equivalent to assumption \eqref{item:A6lipprime} and it is well defined by the choice of $M$ and $\varepsilon$. This concludes the proof.
\end{proof}

We now provide two explicit optimal control strategies satisfying both Proposition \ref{prop:A6prop} and Assumption \ref{item:contalpha}-\ref{item:contalpha1}, so that these controls satisfy Assumption \ref{item:contalpha} in full.

\begin{remark}[Linear-cost optimal control]\label{rem:linearcontrol}
In the special case $\varphi(z) \defin M_1 z$ then we can easily compute $\alpha^*$ from the first order condition \eqref{eq:firstoder} to get
\begin{align}\label{eq:linearcontrol}
\alpha^*(t,x,p;\nu) = -p - M_1\Theta(\nu),
\end{align}
which satisfies the assumptions of Proposition \ref{prop:A6prop} for $M_1$ small enough. Furthermore, if $p \in\Hoelder{\beta}{\beta/2}{\overline{Q}_T;\R^d}$ and $\nu_t \in \mathcal{C}^{\beta/2}([0,T];\mathcal{M}_1(\overline{\Omega}\times A))$, then by Proposition \ref{prop:holderintegral}, the control $\alpha^*$ satisfies Assumption \ref{item:contalpha}-\ref{item:contalpha1}.

In this particular case we can actually solve the self-referencing hidden in \eqref{eq:linearcontrol} by means of a simple computation. Indeed, for given $p$ and $m$ we have
\begin{align*}
\alpha^*&(t,x,p_t(x);\Gamma(p_t,m_t))  = -p_t(x) - M_1\Theta(\Gamma(p_t,m_t)) \\
& \Longleftrightarrow \alpha^*(t,x,p_t;\Gamma(p_t,m_t))m_t(x) = -p_t(x)m_t(x) - M_1\Theta(\Gamma(p_t,m_t))m_t(x) \\
& \Longleftrightarrow \int_{\overline{\Omega}}\alpha^*(t,x,p_t;\Gamma(p_t,m_t))m_t(x)dx = -\int_{\overline{\Omega}}p_t(x)m_t(x)dx -M_1\Theta(\Gamma(p_t,m_t))\int_{\overline{\Omega}}m_t(x)dx \\
& \Longleftrightarrow \Theta(\Gamma(p_t,m_t)) = -\int_{\overline{\Omega}}p_t(x)m_t(x)dx -M_1m_t(\overline{\Omega})\Theta(\Gamma(p_t,m_t)) \\
& \Longleftrightarrow \Theta(\Gamma(p_t,m_t)) = -\frac{1}{1 +M_1m_t(\overline{\Omega}) }\int_{\overline{\Omega}}p_t(x)m_t(x)dx,
\end{align*}
hence we get the closed-loop form
\begin{align}\label{eq:linearcontrolexplicit}
\alpha^*(t,x,p_t(x);\Gamma(p_t,m_t)) = -p_t(x) + \frac{M_1}{1 +M_1m_t(\overline{\Omega}) }\int_{\overline{\Omega}}p_t(x)m_t(x)dx.
\end{align}
This representation clearly shows that, along the typical adjoint term $-p_t$, in this modeling scenario the average adjoint with respect to the mass of players contributes in determining the optimal choice. Notice that this extra term vanishes as soon as $m_t$ does.
\end{remark}

\begin{remark}[Exponential-cost optimal control]\label{rem:expcontrol}
If we assume that $\varphi(z) \defin M_1 \exp(M_2 z)$ then the first order equation \eqref{eq:firstoder} reads
\begin{align}\label{eq:productlog}
p + \alpha + M_1M_2\exp\left(M_2\alpha\cdot\Theta(\nu)\right)\Theta(\nu) = 0,
\end{align}
whose solution by Proposition \ref{prop:productlog} is given by
\begin{align*}
\alpha^*(t,x,p;\nu) = -p - M_1M_2 \exp\left(M_2 W\left(M_1M_2^2|\Theta(\nu)|^2 p\cdot \Theta(\nu) \exp\left(-M_1M_2p\cdot \Theta(\nu)\right) \right) \right) \Theta(\nu).
\end{align*}
Notice how this functional form resembles \eqref{eq:linearcontrol}, except for the multiplier for $\Theta(\nu)$. This solution satisfies Assumption \ref{item:contalpha}--\ref{item:contalpha4} for the appropriate choices of $M_1$ and $M_2$, thanks to Proposition \ref{prop:A6prop}. In addition to this, by invoking Lemma \ref{le:lipbasic} and Proposition \ref{prop:holderintegral}, if $p \in\Hoelder{\beta}{\beta/2}{\overline{Q}_T;\R^d}$ and $\nu_t \in \mathcal{C}^{\beta/2}([0,T];\mathcal{M}_1(\overline{\Omega}\times A))$ then $\alpha^*$ satisfies the H\"older regularity conditions of Assumption \ref{item:contalpha}-\ref{item:contalpha1}.
\end{remark}

\subsection{Sufficient conditions for Assumption \ref{item:hoelderH}}\label{sec:assumption7alt}

Before moving to actual applications of our modeling framework, we show how Assumption \ref{item:hoelderH} can be satisfied. We first introduce the following definition.

\begin{definition}
	Let $Y \subseteq \R^n$ for some $n \in \N$. A function $f:[0,T]\times\overline{\Omega}\times \R^d\times\mathcal{M}_1(\overline{\Omega}\times A)\rightarrow Y$ is \textit{bounded along fixed points} iff there exists a continuos function $\mathcal{H}_f:\R\rightarrow\R$ such that for every $(p,m)\in\Hoelder{\beta}{\beta/2}{\overline{Q}_T;\R^d}\times \Hoelder{\beta}{\beta/2}{\overline{Q}_T;\R}$ it holds
	\begin{align*}
	\|f(t,x,p(t,x);\Gamma(p_t,m_t))\|_{\Hoelder{\beta}{\beta/2}{\overline{Q}_T;Y}}\leq \mathcal{H}_f\left( \|p\|_{\Hoelder{\beta}{\beta/2}{\overline{Q}_T;\R^d}} + \|m\|_{\Hoelder{\beta}{\beta/2}{\overline{Q}_T;\R}} \right).
	\end{align*}
\end{definition}

In order to give concreteness to the above definition, we immediately give an example of an optimal control strategy which is bounded along fixed points.

\begin{lemma}\label{le:linearcontrolbounded}
	The linear-cost optimal control of Remark \ref{rem:linearcontrol} is bounded along fixed points.
\end{lemma}
\begin{proof}
	We start from the equivalent formulation of the linear-cost optimal control given by \eqref{eq:linearcontrolexplicit}. Since all quantities inside integrals are bounded, for every $r,s$ such that $2r+s\leq\beta$ by Leibniz integral rule and Fa\`a di Bruno's formula we get the identity
	\begin{align*}
	D^r_tD^s_x &\alpha^*(t,x,p_t(x),\Gamma(p_t,m_t))  = -D^r_tD^s_xp_t(x) + D^r_t\left(\frac{M_1}{1 +M_1m_t(\overline{\Omega}) }\int_{\overline{\Omega}}p_t(x)m_t(x)dx\right) \\
	& = -D^r_tD^s_xp_t(x) \\
	& \qquad+ \sum_{k = 0}^r \binom{r}{k}\left(D^k_t\left(\frac{M_1}{1 +M_1m_t(\overline{\Omega})}\right)\int_{\overline{\Omega}}p_t(x)m_t(x)dx + \frac{M_1}{1 +M_1m_t(\overline{\Omega})}D^k_t\int_{\overline{\Omega}}p_t(x)m_t(x)dx \right)\\
	& = -D^r_tD^s_xp_t(x) \\
	& \qquad+ \sum_{k = 0}^r \binom{r}{k}\Bigg(\sum^k_{j=0}\frac{(-1)^j j! M_1^{j+1}}{(1 +M_1m_t(\overline{\Omega}))^{j+1}}B_{k,j}\left(D_tm_t(\overline{\Omega}), \ldots, D^{k-j+1}_tm_t(\overline{\Omega})\right)\int_{\overline{\Omega}}p_t(x)m_t(x)dx \\
	&\qquad + \frac{M_1}{1 +M_1m_t(\overline{\Omega})}\sum^k_{j=0}\binom{k}{j}\int_{\overline{\Omega}}D^j_tp_t(x)D^{k-j}_tm_t(x)dx \Bigg),
	\end{align*}
	where $B_{k,j}$ denotes the $j$-th summand of the $k$-th complete exponential Bell polynomial.
	
	Since all the quantities involved on the right-hand side involve integrals of $p_t$ and $m_t$, it can be easily shown that
	\begin{align*}
	\|D^r_tD^s_x \alpha^*(t,x,p_t(x),\Gamma(p_t,m_t))\|_{\infty}  \leq \mathcal{H}_{r,s}\left(\|p\|_{\Hoelder{\beta}{\beta/2}{\R^d;\R^d}} + \|m\|_{\Hoelder{\beta}{\beta/2}{\R^d;\R}}\right),
	\end{align*}
	for some continuous function $\mathcal{H}_{r,s}:\R\rightarrow\R$. In a similar fashion, the above identity allows to bound the H\"older norm of $D^r_tD^s_x \alpha^*$ by a continuous function of the H\"older norms of $p$ and $m$. Summing all contributions for all $r, s$ such that $2r+s\leq\beta$, we get the result.
\end{proof}

Assumption \ref{item:hoelderH} can be simply reformulated by saying that the Hamiltonian $H$ as well as its derivative $D_pH$ are bounded along fixed points.
The following result shows that if the optimal control is bounded along fixed points and if the functionals of the system preserve boundedness along fixed points, then the Hamiltonian satisfies Assumption \ref{item:hoelderH}.

\begin{proposition} \label{eq:prophoeldeH}
	
	Assume that
	\begin{enumerate}
		
		\item The function $\alpha^*:[0,T]\times\overline{\Omega}\times A \times\mathcal{M}_1(\overline{\Omega}\times A)\rightarrow A$ is bounded along fixed points,
		
		\item the function $b:[0,T]\times\overline{\Omega}\times A\times\mathcal{M}_1(\overline{\Omega}\times A)\rightarrow\R^d$ is such that, if $\alpha^*(t,x,p;\nu)$ is bounded along fixed points, then also $\overline{b}(t,x,p;\nu) \defin b(t,x,\alpha^*(t,x,p;\nu);\nu)$ is,
		
		\item The function $\mathcal{L}:[0,T]\times\overline{\Omega}\times A \times\mathcal{M}_1(\overline{\Omega}\times A)\rightarrow\R$ is such that, if $\alpha^*(t,x,p;\nu)$ is bounded along fixed points, then also $\overline{\mathcal{L}}(t,x,p;\nu) \defin \mathcal{L}(t,x,\alpha^*(t,x,p;\nu);\nu)$ is.
		
	\end{enumerate}
	
	Then the Hamiltonian $H$ satisfies Assumption \ref{item:hoelderH}.
\end{proposition}
\begin{proof}
	Let $\mathcal{H}_{b}, \mathcal{H}_{\mathcal{L}}:\R\rightarrow\R$ be the continuous function associated to the functions $b$ and $\mathcal{L}$, respectively. From the definition \eqref{def:hamiltonian} we get
	\begin{align*}
	\|H&(t,x,p_t(x);\Gamma(p_t,m_t))\|_{\Hoelder{\beta}{\beta/2}{\R^d;\R}} \leq \\
	& \leq \|p \|_{\Hoelder{\beta}{\beta/2}{\R^d;\R^d}}\| b(t,x,\alpha^*(t,x,p_t(x);\Gamma(p_t,m_t));\Gamma(p_t,m_t))\|_{\Hoelder{\beta}{\beta/2}{\R^d;\R^d}} \\
	& \qquad\qquad\qquad+\| \mathcal{L}(t,x,\alpha^*(t,x,p_t(x);\Gamma(p_t,m_t));\Gamma(p_t,m_t))\|_{\Hoelder{\beta}{\beta/2}{\R^d;\R}}\\
	&\leq \|p \|_{\Hoelder{\beta}{\beta/2}{\R^d;\R^d}} \mathcal{H}_{b}\left(\|p\|_{\Hoelder{\beta}{\beta/2}{\R^d;\R^d}} + \|m\|_{\Hoelder{\beta}{\beta/2}{\R^d;\R}}\right) \\
	& \qquad\qquad\qquad+\mathcal{H}_{\mathcal{L}}\left(\|p\|_{\Hoelder{\beta}{\beta/2}{\R^d;\R^d}} + \|m\|_{\Hoelder{\beta}{\beta/2}{\R^d;\R}}\right),
	\end{align*}
	which shows that $H$ is bounded along fixed points by means of the function $\mathcal{H}_1(x)\defin x\mathcal{H}_b(x) + \mathcal{H}_{\mathcal{L}}(x)$. Since $D_pH$ is trivially bounded along fixed points by the function $\mathcal{H}_2(x)\defin \mathcal{H}_b(x)$, the statement is proven.
\end{proof}

We now show examples of functions that preserve boundedness along fixed points (and thus can be employed in the definition of $b$ and $\mathcal{L}$ to give rise to a Hamiltonian for which Proposition \ref{eq:prophoeldeH} holds).

\begin{lemma} \label{le:thetalip}
	If $\alpha^*(t,x,p;\nu)$ is bounded along fixed points, then also $\Theta(\nu)$ is.
\end{lemma}
\begin{proof}
	We start by noticing that, since $\alpha^*(t,x,p(t,x);\Gamma(p_t,m_t)$ is a bounded $\Hoelder{\beta}{\beta/2}{\overline{Q}_T;\R}$-function for any $(p,m)\in\Hoelder{\beta}{\beta/2}{\overline{Q}_T;\R^d}\times \Hoelder{\beta}{\beta/2}{\overline{Q}_T;\R}$, then by the Leibniz integral rule it holds
	\begin{align*}
	D^r_t\Theta(\Gamma(p_t, m_t)) & = D^r_t\int_{\overline{\Omega}\times A}\alpha d\Gamma(p_t, m_t)(y,\alpha)\\
	& = D^r_t\int_{\overline{\Omega}}\alpha^*(t,y,p(t,y);\Gamma(p_t, m_t)) m_t(y) dy\\
	& = \int_{\overline{\Omega}}D^r_t\left(\alpha^*(t,y,p(t,y);\Gamma(p_t, m_t)) m_t(y)\right)dy\\
	& = \sum_{k = 1}^r \binom{r}{k}\int_{\overline{\Omega}} D^k_t\alpha^*(t,y,p(t,y);\Gamma(p_t, m_t)) D^{r-k}_t m_t(y)dy.
	\end{align*}
	We can therefore argue as in the proof of Lemma \ref{le:linearcontrolbounded} to conclude that $\Theta(\nu)$ is bounded along fixed points.
\end{proof}

The following statements are easy corollaries of the above result.

\begin{corollary}\label{cor:fixpointscost}
	Let $\phi:A\rightarrow\R^d$ be a smooth function. Then if $\alpha^*(t,x,p;\nu)$ is bounded along fixed points, also $\overline{f}(t,x,p;\nu) \defin \phi(\alpha^*(t,x,p;\nu)\cdot \Theta(\nu))$ is.
\end{corollary}

\begin{corollary}\label{cor:fixpointsdynrates}
	Let $\phi:\R\rightarrow\R^d$ be a smooth function. Then if $\alpha^*(t,x,p;\nu)$ is bounded along fixed points, also $\overline{f}(x;\nu) \defin (1 + \phi(\Theta(\nu))x$ is.
\end{corollary}

Another corollary is the following result which shows that the convolution dynamics preserve boundedness along fixed points.

\begin{corollary}
	Let $K:\R^d\rightarrow\R^d$ be a smooth bounded function and set
	$$f(t,x;\nu) \defin (K*\pi_{1\sharp}\nu)(x).$$
	Then if $\alpha^*(t,x,p;\nu)$ is bounded along fixed points, also $f(t,x;\nu)$ is.
\end{corollary}
\begin{proof}
	By $\pi_{1\sharp}\Gamma(p_t, m_t) = m_t$ and the properties of the convolution operator follows that if $m \in \Hoelder{\beta}{\beta/2}{\overline{Q}_T;\R}$ then
	\begin{align*}
	D^r_tD^s_x (K*\pi_{1\sharp}\Gamma(p_t, m_t) )(x) & = D^r_tD^s_x (K*m_t )(x)\\
	& = \int_{\overline{\Omega}}K(x-y) m_t(y) dy\\
	& = \int_{\overline{\Omega}}D^s_xK(x-y) D^r_tm_t(x-y) dy\\
	& =  (D^s_x K)*\left(D^r_tm_t\right) (x),
	\end{align*}
	where we passed derivation inside the integral because of Leibniz integral rule. This easily implies that
	\begin{align*}
	\|K*\pi_{1\sharp}\Gamma(p_t, m_t) \|_{\Hoelder{\beta}{\beta/2}{\overline{Q}_T;\R^d}}\leq \|K\|_{\mathcal{C}^\infty}|\overline{\Omega}| \|m_t \|_{\mathcal{C}^{\beta/2}{([0,T];\R)}},
	\end{align*}
	which in turn yields the statement.
\end{proof}

\subsection{Refinancing dynamics}\label{sec:refinancing}

We now pass to show the first application of our modeling framework. 

Assume that $d=1$ and $\Omega = (-1,1)$, and let us consider a continuum of indistiguishable firms with $X_t \in [-1,1]$ denoting the face value of single-period debt of each firm. Notice that being indistiguishable implies that the impact of the individual choice on debt value must be infinitesimal. To simplify the setting, we assume that the level $X_t = 1$ corresponds to the default of the firm (e.g., for the lack of collateral to pledge in return), while $X_t = -1$ means that the firm quits the debt market since it has enough savings for the rest of its lifetime.

At each $t$, the firm decides the adjustment $\alpha^*$ of its debt level according to the following minimization problem
\begin{align*}
\min_{\alpha\in \mathcal{U}} J(\mattia{0},x_0,\alpha;\nu) = \mathbb{E}\left[\int^{T\wedge\tau}_0 \left(\ell(t,X_t) + g(\alpha_t,\nu_t) + \varepsilon|\alpha_t|^2 \right)dt + \psi(X_{T\wedge\tau},m_{T\wedge \tau})\right],
\end{align*}
where the term $|\alpha_t|^2$ should be considered as a \textit{deadweight adjustment cost}. The stopping time $\tau$ can be interpreted as the exit time from the debt market (if an agent could exit only in the case $X_t = 1$ it could be interpreted as a default time, but in the present case we allow to exit also for the case $X_t = -1$). Furthermore, the infinitesimal firm rolls over its entire stock of debt by paying a rate of return $R(\nu_t)$ which depends on the current demand for debt in the market, i.e., for some $\rho:\R\rightarrow\R$ sufficiently smooth we have
\begin{align*}
R(\nu) \defin \rho\left(\int_{\overline{\Omega}\times A} \alpha d\nu(x,\alpha)\right) = \rho(\Theta(\nu)).
\end{align*}
Similarly, we shall prescribe that the function $g:A\times\mathcal{M}_1(\overline{\Omega}\times A)\rightarrow \R$ is given by
\begin{align*}
g(\alpha,\nu) \defin \varphi\left(\alpha\int_{\overline{\Omega}\times A}\tilde{\alpha}d\nu(\tilde{x},\tilde{\alpha})\right) = \varphi(\alpha\Theta(\nu)),
\end{align*}
with $\varphi:\R\rightarrow \R$ smooth, increasing and convex with $\varphi'\geq B$ for some $B \geq 0$. With this cost, it is advantageous for the individual firm to \say{go against} the average demand for debt. The motivation is that it is better for a firm to increase its stock of debt (i.e., $\alpha > 0$) while there is an excess supply of funds, corresponding to the condition
$$\Theta(\nu) = \int_{\overline{\Omega}\times A}\tilde{\alpha}d\nu(\tilde{x},\tilde{\alpha}) < 0,$$
than while there is excess demand.

This setting yields the minimization problem
\begin{align*}
\min_{\alpha\in \mathcal{U}} J(\mattia{0},x_0,\alpha;\mu) = \mathbb{E}\left[\int^{T\wedge\tau}_0 \left(\ell(t,X_t) + g(\alpha_t,\mu_t) + \varepsilon|\alpha_t|^2 \right)dt + \psi(X_{T\wedge\tau},m_{T\wedge \tau})\right]
\end{align*}
subject to the stopped SDE
\begin{align*}
\left\{\begin{aligned}
d X_t &= \left((1+R(\mu_t))X_t + \alpha^*(t,X_t,Du(t,X_t);\mu_t)\right)dt + 2\sqrt{\sigma}dW_t & \quad \text{ for } t \in [0,T\wedge\tau],\\
X_t & = X_{\tau} & \quad \text{ for } t \in [T\wedge\tau,T]
\end{aligned}\right.
\end{align*}
and the fixed point relationship for $\mu_t$ given by \eqref{eq:fixpointmu}.

It is quite easy to show that we fall under the general framework outlined in Section \ref{sec:assumption} as soon as we choose $\varphi(z) = M_1z$ for the appropriate value of $M_1>0$. As Assumptions \ref{item:compactset}--\ref{item:inverseb} are trivially satisfied, let us see how we can employ the sufficient conditions of the previous sections to show that also Assumptions \ref{item:monotoneL}-\ref{item:hoelderH} hold true.

Firstly, Assumption \ref{item:monotoneL} holds as soon as we show that Proposition \ref{prop:monotoneLprop} applies to our choice of $g$: since it holds
\begin{align*}
g(\alpha,\nu_1) - g(\alpha,\nu_2) & = \varphi\left(\alpha\int_{\overline{\Omega}\times A}\tilde{\alpha}d\nu_1(\tilde{x},\tilde{\alpha})\right) - \varphi\left(\alpha\int_{\overline{\Omega}\times A}\tilde{\alpha}d\nu_2(\tilde{x},\tilde{\alpha})\right)\\
& = M_1\left(\alpha\int_{\overline{\Omega}\times A}\tilde{\alpha}d\nu_1(\tilde{x},\tilde{\alpha}) - \alpha\int_{\overline{\Omega}\times A}\tilde{\alpha}d\nu_2(\tilde{x},\tilde{\alpha})\right)\\
& = M_1\alpha\int_{\overline{\Omega}\times A} \tilde{\alpha} d(\nu_1 - \nu_2)(\tilde{x},\tilde{\alpha}).
\end{align*}
then we have \eqref{eq:suffmonotone} for $F(x,\alpha) = \alpha\sqrt{M_1}$.

Secondly, Proposition \ref{prop:A6prop} holds true as soon as we choose $M_1 \leq 2\varepsilon$, so that in this case also Assumption \ref{item:contalpha} holds.

Finally, as we are in the linear-cost optimal control case, we can employ all the results of Section \ref{sec:assumption7alt} to establish the validity of Assumption \ref{item:hoelderH}, in particular Corollaries \ref{cor:fixpointscost} and \ref{cor:fixpointsdynrates}.

\subsection{Evacuation of pedestrians}\label{ss:evac_ped}

Let $d = 2$, $\Omega\subset\R^2$ be an open subset and $A = B_M(0)$ for some $M>0$. We consider a continuum of infinitesimal pedestrian whose dynamics is given by
$$dX_t = K*m_t(X_t)dt + \alpha^*dt + 2\sqrt{\sigma}dW_t,$$
for some smooth convolution kernel $K:\R^d\rightarrow\R^d$ (like the smoothed Hegselmann-Krause one). The optimal velocity $\alpha^*$ should be chosen according to two main principles: firstly, pedestrians should try to avoid congestions, which can be done by minimizing the functional
\begin{align*}
\int_0^T Q*m_t(X_t) dt = \int_0^T\int_{\overline{\Omega}}Q(X_t-\tilde{x})dm_t(\tilde{x})dt,
\end{align*}
where $Q:\R^d\rightarrow\R$ is a decreasing radial function (see \cite{bongini2016optimal} for an analysis of this cost term in the context of pedestrian dynamics) satisfying $\inf_{\overline{\Omega}} Q\geq B$ for some $B>0$. Secondly, as argued in \cite{albi2015invisible}, pedestrian have a natural attitude to follow their mates, a tendency that can be reproduced by minimizing the quantity
\begin{align*}
\int_0^T \varphi(\alpha\cdot\Theta(\nu))dt = \int_0^T\varphi\left(\alpha\cdot\int_{\overline{\Omega}\times A}\tilde{\alpha}d\nu(\tilde{y},\tilde{\alpha})\right)dt,
\end{align*}
where $\varphi:\R\rightarrow\R$ is a strictly convex, decreasing functional with $\varphi' \geq -B/M^2$. Indeed, if this quantity is minimized, then the angle between the optimal velocity $\alpha^*$ and the average velocity chosen by the rest of the agents is going to be \say{small}. As a result, this model has the advantage of being able to treat alignment of agents even if it is a first-order model.

Putting things together, we obtain the minimization problem
\begin{align*}
\min_{\alpha \in \mathcal{U}} J(\mattia{0},x_0,\alpha;\mu)  =  \mathbb{E}\left[\int^{T\wedge\tau}_0 \left(Q*\pi_{1\sharp}\mu_t(X_t) +\varphi(\alpha\cdot\Theta(\nu)) + \varepsilon|\alpha_t|^2 \right)dt + \psi(X_{T\wedge\tau},m_{T\wedge\tau})\right]
\end{align*}
subject to the stopped SDE
\begin{align*}
\left\{\begin{aligned}
d X_t & = K*m_t(X_t)dt + \alpha^*(t,X_t,Du(t,X_t);\mu_t)dt + 2\sqrt{\sigma}dW_t & \quad \text{ for } t \in [0,T\wedge\tau],\\
X_t & = X_{\tau} & \quad \text{ for } t \in [T\wedge\tau,T],
\end{aligned}\right.
\end{align*}
and the fixed point relationship for $\mu_t$ given by \eqref{eq:fixpointmu}.

Arguing similarly to Section \ref{sec:refinancing}, we can easily show that the general framework presented above applies to this model as well as soon as we choose $\varphi(z) = -M_1z$ with $M_1 \leq B/M^2$. We shall only show that Proposition \ref{prop:monotoneLprop} holds true, since we trivially have
\begin{align*}
M_1\alpha&\cdot\int_{\overline{\Omega}\times A}\tilde{\alpha}d\nu_1(\tilde{x},\tilde{\alpha})+\int_{\overline{\Omega}}Q(x-\tilde{x})d\pi_{1\sharp}\nu_1(\tilde{x})- M_1\alpha\cdot\int_{\overline{\Omega}\times A}\tilde{\alpha}d\nu_2(x,\tilde{\alpha}) -\int_{\overline{\Omega}}Q(x-\tilde{x})d\pi_{1\sharp}\nu_2(\tilde{x})\\
& = M_1\left(\alpha\cdot\int_{\overline{\Omega}\times A}\tilde{\alpha}d(\nu_1-\nu_2)(x,\tilde{\alpha})\right)+\int_{\overline{\Omega}}Q(x-\tilde{x})d\pi_{1\sharp}(\nu_1-\nu_2)(\tilde{x})\\
&\geq -M_1M^2\frac{|\alpha|}{M}\int_{\overline{\Omega}\times A}\frac{|\tilde{\alpha}|}{M}d(\nu_1-\nu_2)(\tilde{x},\tilde{\alpha}) + B\int_{\overline{\Omega}\times A}d(\nu_1-\nu_2)(\tilde{x},\tilde{\alpha})\\
& \geq -M_1M^2\int_{\overline{\Omega}\times A}d(\nu_1-\nu_2)(\tilde{x},\tilde{\alpha})+B\int_{\overline{\Omega}\times A}d(\nu_1-\nu_2)(\tilde{x},\tilde{\alpha}) \geq 0,
\end{align*}
where we used the fact that $M \geq |\alpha|$. However the above chain of inequalities reduces to \eqref{eq:suffmonotone} for the trivial choice $F \equiv 0$.

\section{Numerical experiments} \label{sec:numerics}

\subsection{The numerical method}

In this article, we have chosen a particle method for simulating system \eqref{eq:mfg}.
The strategy consists in describing a population, composed of a great number of individuals,
by means of a reduced set of numerical particles, which obey to the action criterion of the problem. 

Let $\varphi_{\varepsilon}$ be an $\varepsilon$-dependent smooth function, such that $\varphi_{\varepsilon}(x)\defin\varphi( \varepsilon^{-1} x)/\varepsilon^d$.
It is often denoted \textit{shape function} because it is used in reconstructing the density of particle starting from the numerical particles. A careful choice of the shape function is crucial for eliminating (or, at least, greatly reducing) the noise of the density profile and for respecting the boundary conditions of the continuous problem.

The first step of the particle method we are going to employ consists in discretizing the initial density $f^{\rm in}$ by means of a set of smooth shape functions centred on the individual positions and controls, in such a way that
\begin{equation} \label{finieps}
f^{\rm in}_{\varepsilon,N_m}(x,\alpha) = \sum_{k=1}^{N_m}  \omega_k \, \varphi_{\varepsilon}(x - x_k^{0})\, \varphi_{\varepsilon}(\alpha - \alpha_k^0),
\end{equation}
where $N_m$ represents the number of numerical particles, $(x_k^{0})_{1\leq k \leq N_m}$ their initial positions and $(\alpha_k^{0})_{1\leq k \leq N_m}$ their initial velocities (i.e. their controls). We underline that each numerical particle is weighted by means of a set of parameters $\omega_k \in \R_+$ (which in our application shall be uniform in $k$). Once the number $N_m$ of numerical particles has been chosen, the initial positions $(x_k^{0})_{1\leq k \leq N_m}$ are sampled according to the initial density $f^{\rm in}$, either in a deterministic way or thanks to a Monte-Carlo procedure. In the first iteration (since the exit cost of our cost functional is going to be set to zero), we choose $(v_k^{0})_{1\leq k \leq N_m}$ such that the particles move in the direction of the boundary point closest to their initial position with velocities of maximum norm.

We introduce a time discretization of step $\Delta t$ so that we set $t^n \defin n \Delta t$. The density of the continuous problem at time $t^n$ is hence
\begin{equation} \label{fneps}
f^{n}_{\varepsilon,N_m}(x,\alpha) = \sum_{k=1}^{N_m}  \omega_k \, \varphi_{\varepsilon}(x - x_k^n)\, \varphi_{\varepsilon}(\alpha - \alpha_k^n),
\end{equation}
where $(x_k^{n})_{1\leq k \leq N_m}$ and $(v_k^{n})_{1\leq k \leq N_m}$ are the positions and the velocities of the numerical particles at time $t^n$. The positions of the numerical particles evolve in time by minimizing the individual cost functional specific to the problem, in a time interval $[n \Delta t,(n+1) \Delta t]$. In the next Section, we will provide an example of such a cost. Once we identify the velocity (that is, the control strategy) which minimizes the individual cost, the numerical particles move according to such velocity.

In order to find the Nash equilibrium solution of the problem, an iterative numerical scheme has been implemented. 
We begin the iteration procedure by considering a simplified problem, and use this solution as the starting point of the iteration scheme. 
First, the admissible velocities are discretized in an appropriated way in order to reduce them to a finite set.
Then, the next steps are implemented as follows. 

In what follows, we denote by $j$ the index labelling the iteration step, by $(x_{k,j}^{n})_{1\leq k \leq N_m}$ and $(\alpha_{k,j}^{n})_{1\leq k \leq N_m}$ the positions and the velocities of the particle labelled with the index $k$ at time $t^n$ at the $j$-th iteration.
At the beginning of the $j$-th iteration, we randomly order the numerical particles, thanks to a sample from the uniform distribution. We then consider the first particle $x^{1}_{1,j}$ (with respect to the random order) at the numerical time $t^1$ and choose $\alpha^1_{1,j}$ by testing the possible costs between $t^1$ and $t^2$ of this particle with respect to the set of discretized admissible velocities, by assuming that the other particles have the positions and the velocities computed at the $(j-1)$-th iteration (i.e., using $(\alpha_{k,j-1}^{0})_{2\leq k \leq N_m}$).
Once we find the minimal cost for the first particle, we update the position and the velocity of such particle, which will be used for the computations of the next step. At the numerical time $t^2$, we consider the second particle $x^2_{2,j}$ (with respect to the random order). We choose $\alpha^2_{2,j}$ by testing the possible costs between $t^2$ and $t^3$ of this particle with respect to the set of discretetized admissible velocities, by assuming that the other particles have the positions and the velocities computed at the $(j-1)$-th iteration, except for the first particle, whose position and velocity have been updated in the previous step (i.e., using $\alpha_{1,j}^1$ and $(\alpha_{k,j-1}^{0})_{3\leq k \leq N_m}$). We then end the procedure either once all the particles have been tested, or when the time horizon of the problem has been reached.

The updated positions and velocities are then the starting point for the $(j+1)$-th iteration.

The procedure ends when the difference between the distribution obtained at the $(j+1)$-th iteration and the distribution obtained at the $j$-th iteration
is below a given threshold.

A peculiar feature of the method is the absence of a stability condition. This could be an advantage for simulating the long-time behaviour of a MFG system, especially in space dimensions higher than one.
Moreover, this technique is very well-suited to treat not only second-order systems of PDEs in space, but also first-order systems. In such case, it allows to manage non-conventional boundary conditions, such as absorbing boundary conditions or specular reflection boundary conditions. We finally underline that particle methods suffer from artificial numerical diffusion much less than finite-difference methods.

On the other hand, in order to have accurate simulations, the number of required numerical particles should be very high and the simulations can be
time-consuming.

In some sense, our procedure, based on successive approximations, mimics a real-life repeated game and is linked to the concepts of
\say{best reply} \cite{MR3268055} and \say{fictitious play} \cite{MR3608094}. We underline that iterative procedures for the numerical implementation of MFG, not based on finite difference methods, have been proposed in the context of non-compulsory vaccination \cite{MR3810807}.

\subsection{Numerical simulations}

In this section, we discuss some numerical results obtained from the implementation of the method described in the previous subsection to the evacuation model of Section \ref{ss:evac_ped}.

We have considered the space domain $\Omega\defin(0,1)$, the control space $A\defin[-.2,.2]$, the convolution kernel $K = 0$, the initial measure $\overline{m}$ as a smoothed version of the step function
\begin{equation*}
f^{\rm in}\defin\left\{
\begin{array}{ll}
0, & x\in(0, 0.5) \\[10pt]
2, & x\in(0.5,1), 
\end{array}
\right.
\end{equation*}
and the cost functional given by
\begin{equation}
\begin{aligned}
	\label{f:int}
	\mathcal{J}(t,x,\alpha;\nu)& \defin  \mathbb{E}\Bigg[\eta \int_{\mattia{t}}^{T\wedge\tau} \left (\left \vert \int_{\Omega} y\, \pi_{1\sharp}\nu_t(y,\tilde{\alpha})dy -x \right\vert +.2\right)^{-1}dt  \\
& \qquad \qquad +\beta  \int_{\mattia{t}}^{T\wedge\tau} \alpha_t\int_{\Omega\times A}{\tilde{\alpha}}\, \nu_t({y},\tilde{\alpha})dy \, d\tilde{\alpha} \, dt + \varepsilon \int_{\mattia{t}}^{T\wedge\tau}|\alpha_t|^2dt\Bigg],
\end{aligned}	
\end{equation}
with $\eta=4$, $\beta=1$, $\varepsilon = 1/2$, and where $\tau$ is the exit time from the domain. Notice that we assume that there is no exit cost, i.e. $\phi=0$. It is easy to show that this functional satisfies the hypotheses of Proposition \ref{prop:A5prop}, Proposition \ref{prop:A6prop} and Proposition \ref{eq:prophoeldeH}. Hence, Assumptions (A) listed in Section \ref{sec:assumption} are satisfied and the pedestrian evacuation model admits a Nash equilibrium.

Notice that the first integral in \eqref{f:int} takes into account the tendency of pedestrians to avoid congestions and the second one the tendency to follow their mates (see Subsection \ref{ss:evac_ped} for more details concerning the meaning of this functional).

In the numerical experiments, we discretized the above quantities in the following way.

First, the initial density used in the numerical experiments is $f^{\rm in}$. We underline that $f^{\rm in}$ has total mass equal to one and its center of mass is located at $x=0.75$.

Then, the simulation has been obtained by using $N=6\times 10^3$ numerical particles.

To discretize the control space $A$, we choose $N_\alpha\in \N$, with $N_\alpha >2$ and denoted by $\Delta\alpha = 0.4/(N_\alpha-1)$. In our simulations we have considered the control set 
$$
A_\alpha \defin \{-0.2, -0.2+\Delta\alpha, \dots, -0.2 + (N_\alpha-1)\Delta\alpha , 0.2\}
$$ 
obtained by discretizing the control space $A$. At each instant, the numerical particles choose their optimal velocity $\alpha^*\in A_\alpha$
in order to minimize the functional \eqref{f:int}.

We have moreover assumed that $N_\alpha=11$ and that $\sigma=2.5\times 10^{-9}$.

The time history of the density is plotted in Figure \ref{fig:1}. We see that the support of the density is split in two disjoints subsets. Then, when the position of the center of mass is modified by the exits from the domain, the new position of the center of mass induces some numerical particles to modify their velocity in order to reduce their global cost.

Moreover, we observe that the numerical boundary conditions are coherent with the problem (i.e., the density vanishes in $x=0$ and in $x=1$).

\begin{figure}[h!]
	\begin{center}
		a)\ \includegraphics[width=0.45\textwidth]{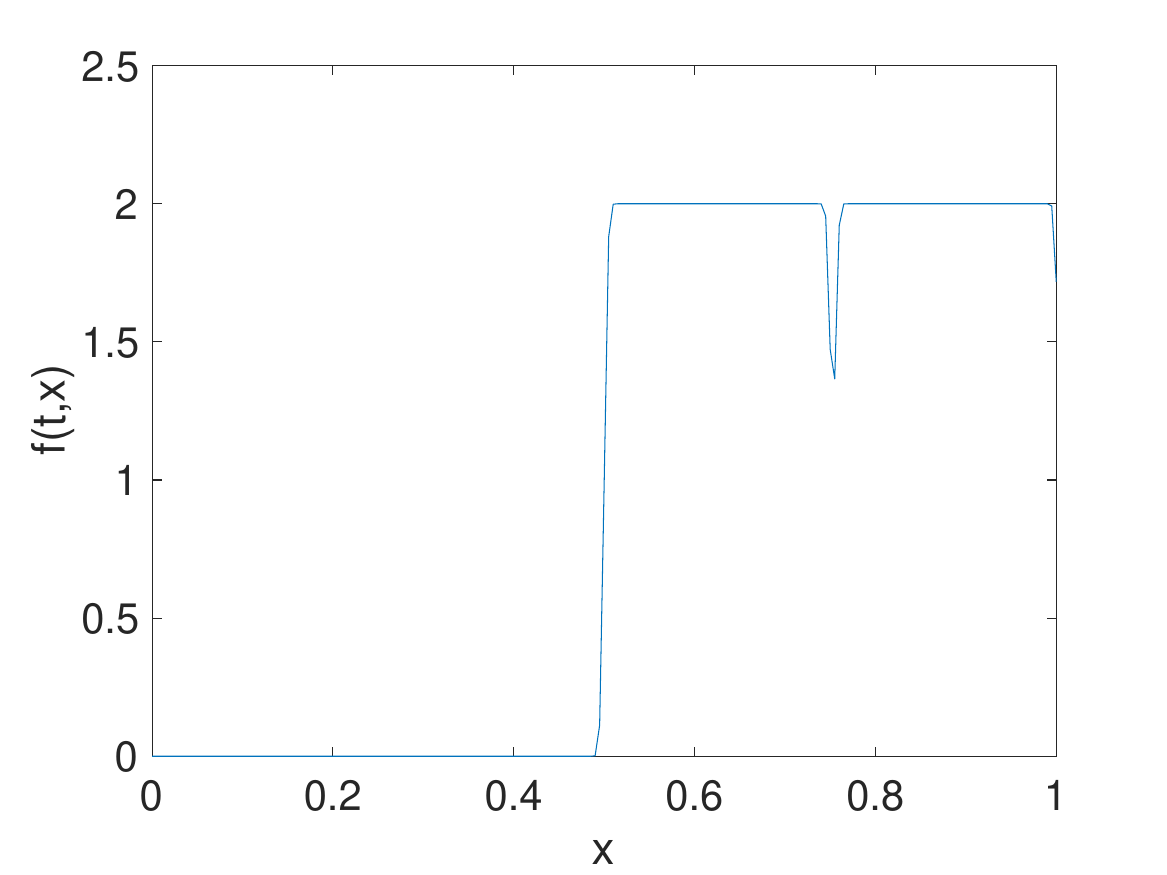}\
		b)\ \includegraphics[width=0.45\textwidth]{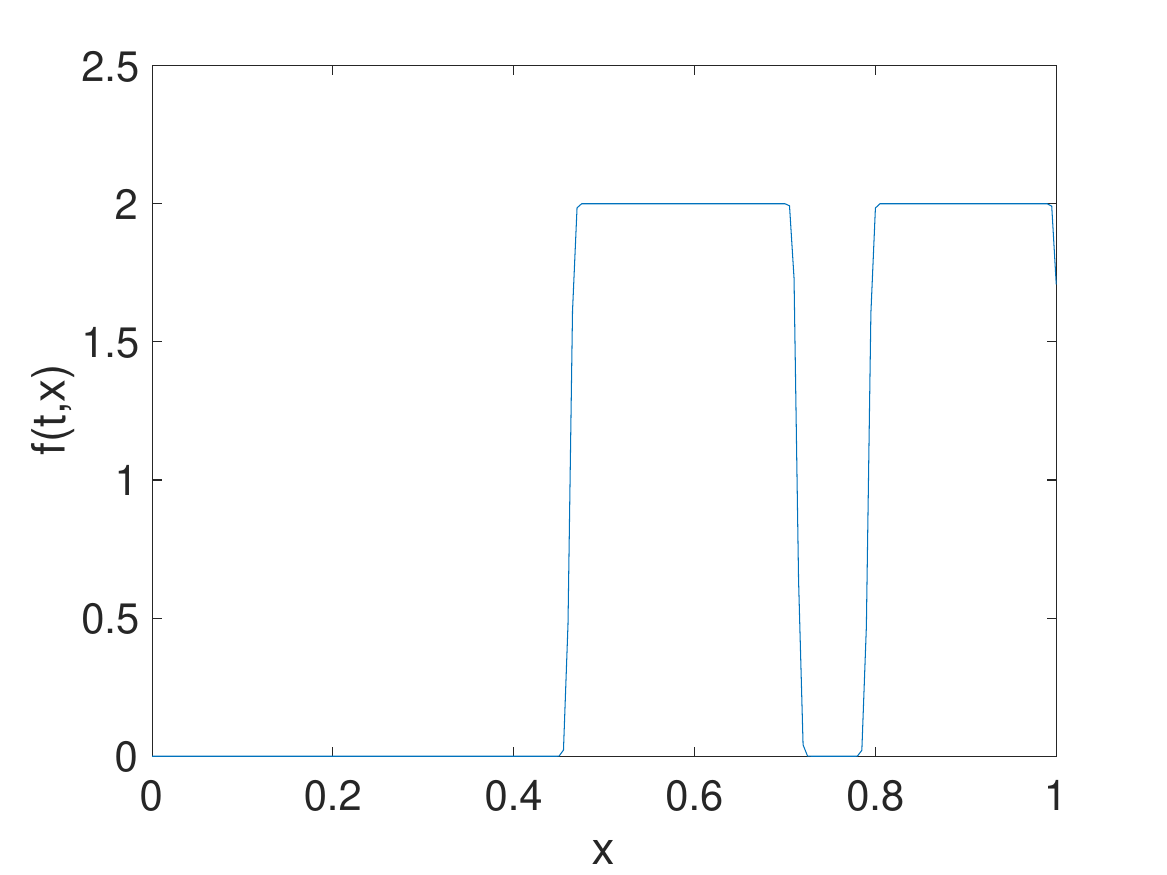}\\[1.2ex]
		c)\ \includegraphics[width=0.45\textwidth]{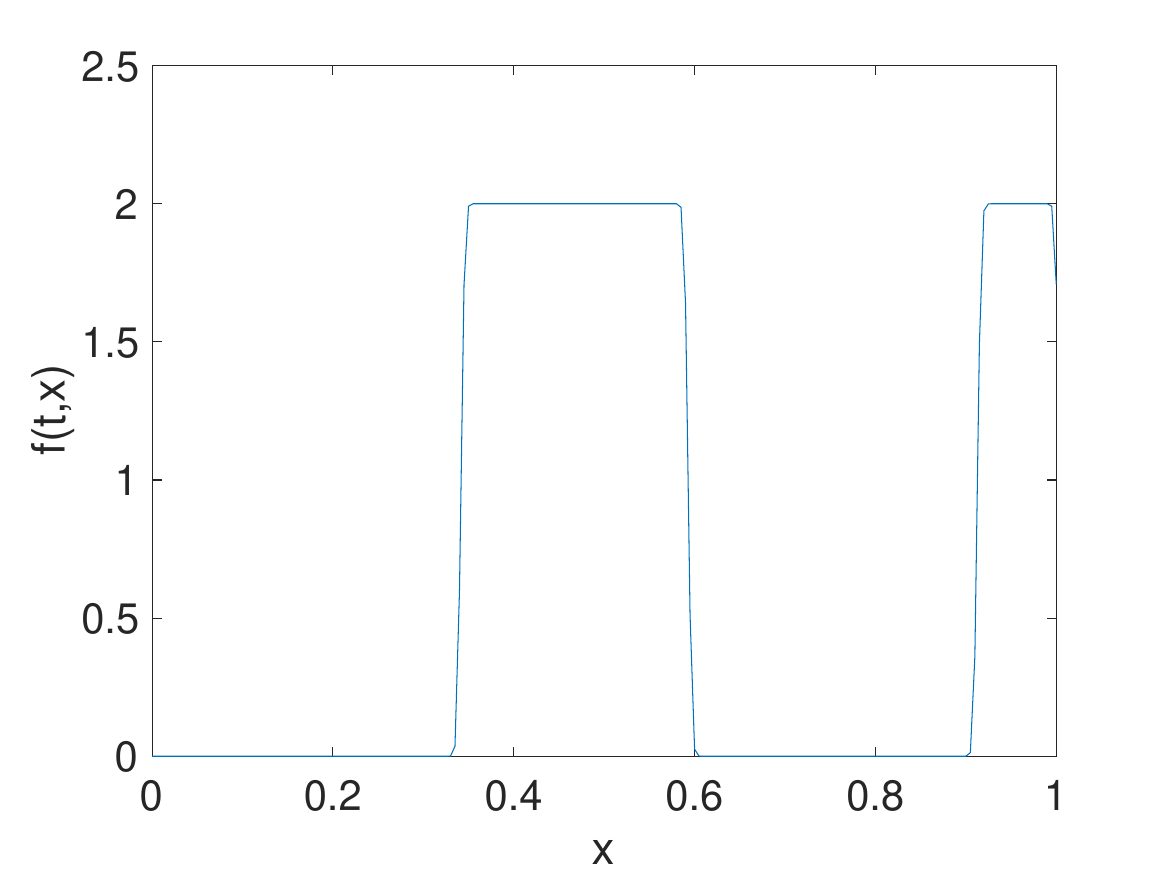}\
		d)\ \includegraphics[width=0.45\textwidth]{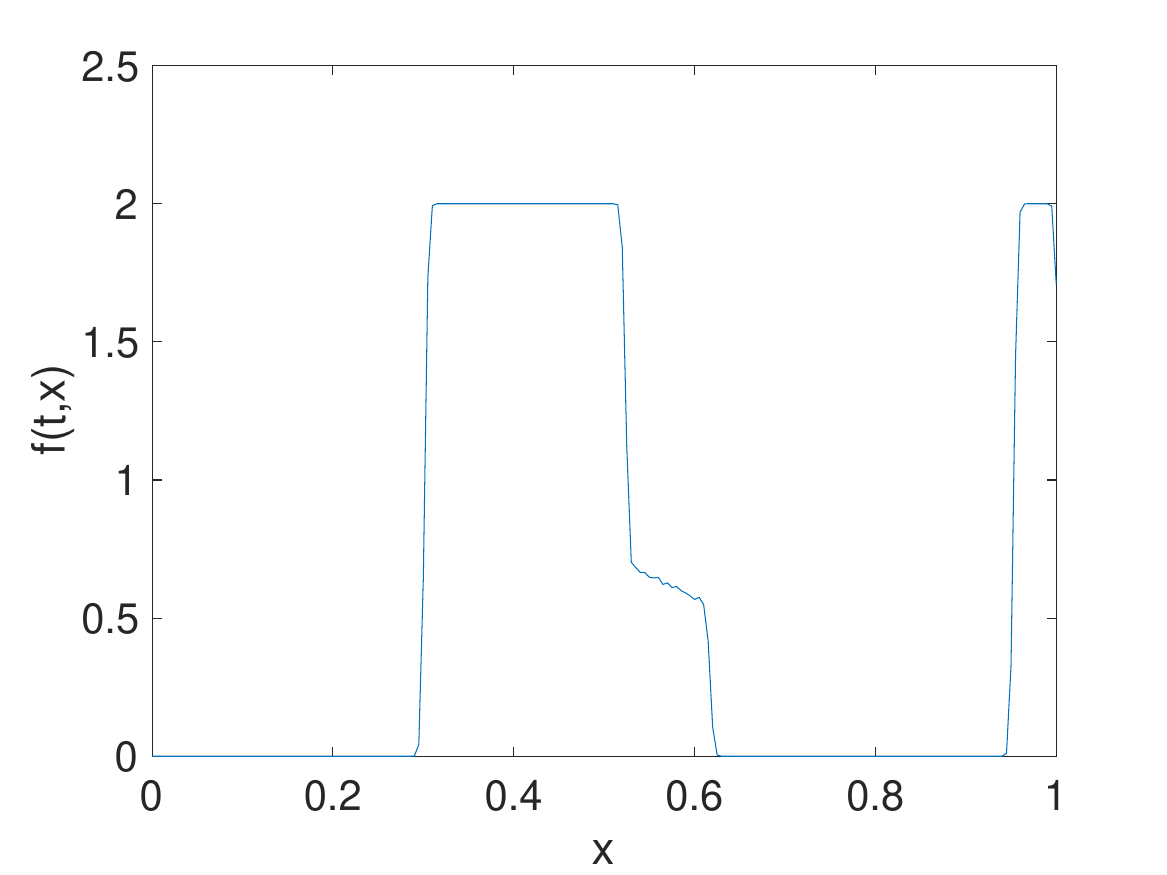}\\ [1.2ex]
		e)\ \includegraphics[width=0.45\textwidth]{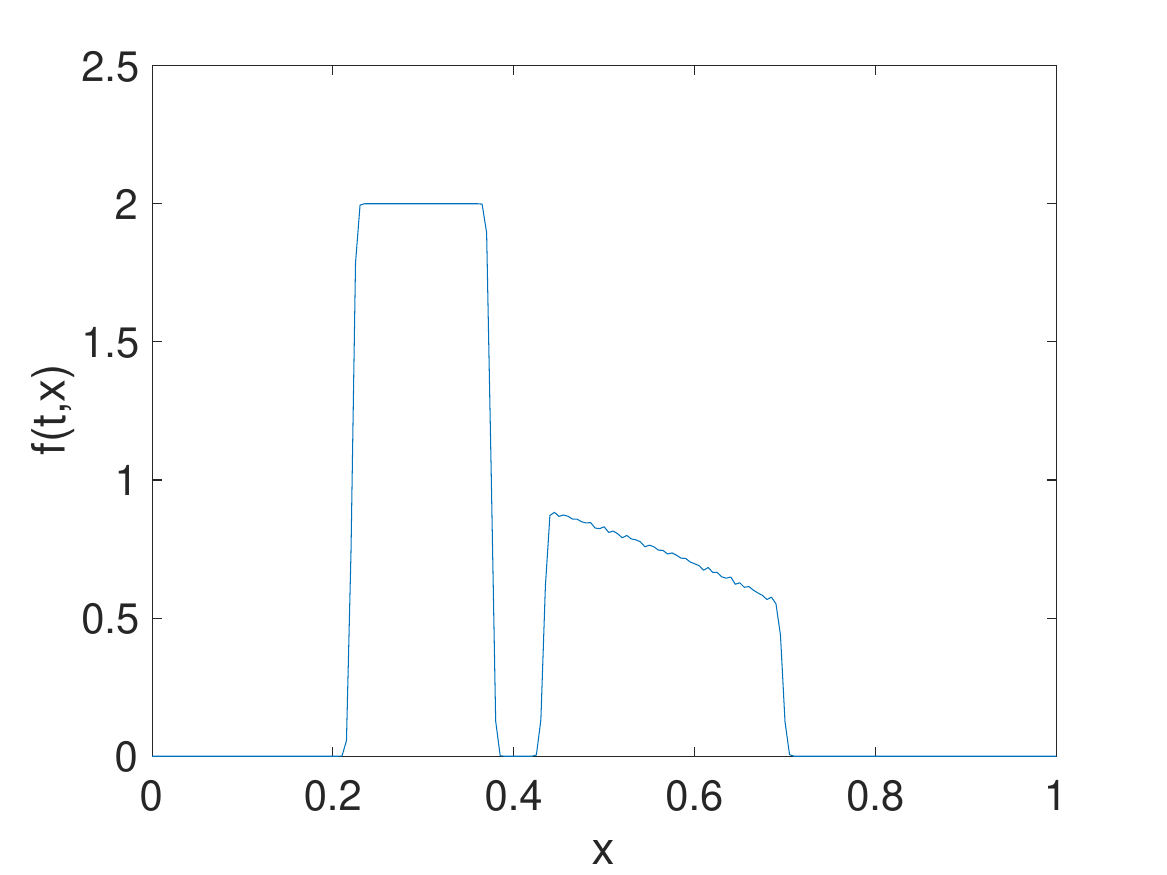}
		f)\ \includegraphics[width=0.45\textwidth]{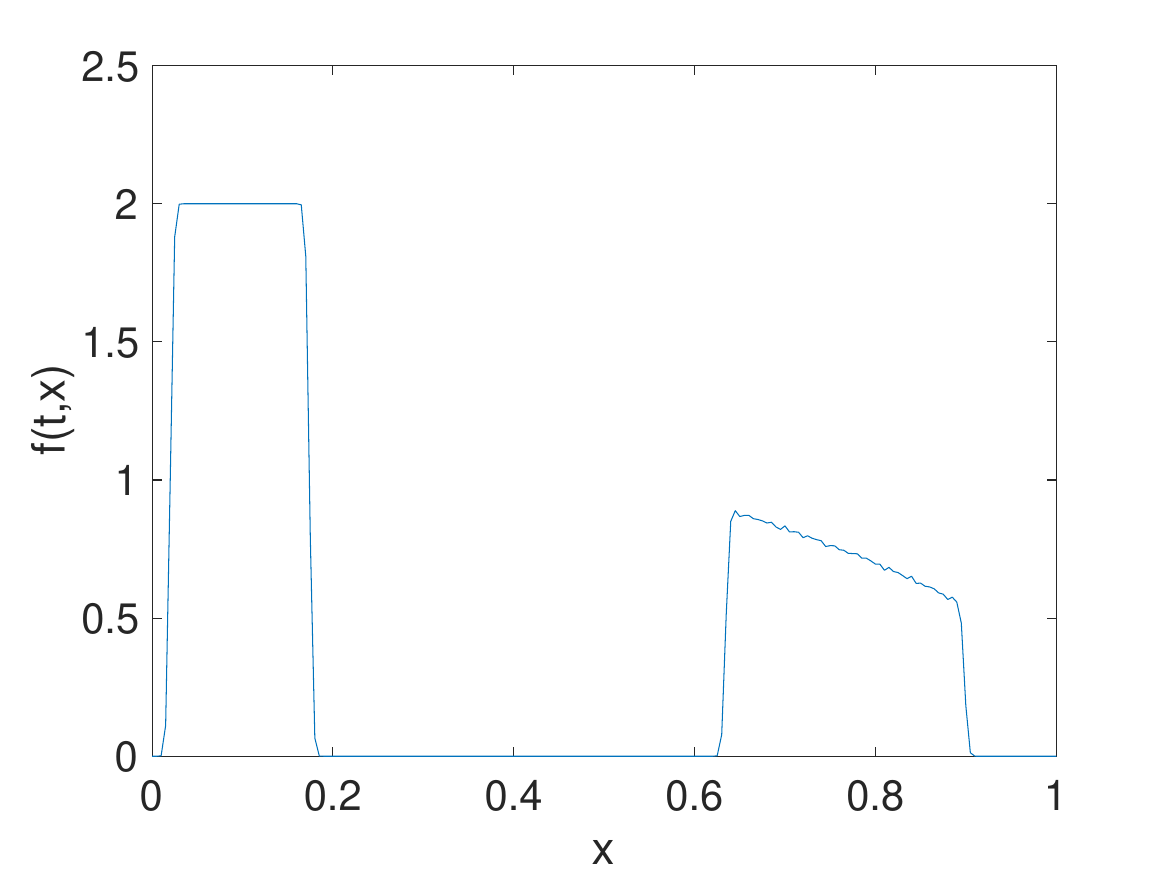}
		\caption{Profiles of the spatial distribution for six different time instants: a) $t=0.012$, b) $t=0.2$, c) $t=0.8$,
			d) $t=1$, e) $t=1.4$ and f) $t=2.4$.} \label{fig:1}
	\end{center}
\end{figure}

Moreover, we have plotted, in Figure \ref{fig:2} (left) the time evolution of the total mass with the data of our simulation and, in Figure \ref{fig:2} (right) the time evolution of the center of mass of the population.

\begin{figure}[!ht]
	\centering
	\includegraphics[width=0.45\textwidth]{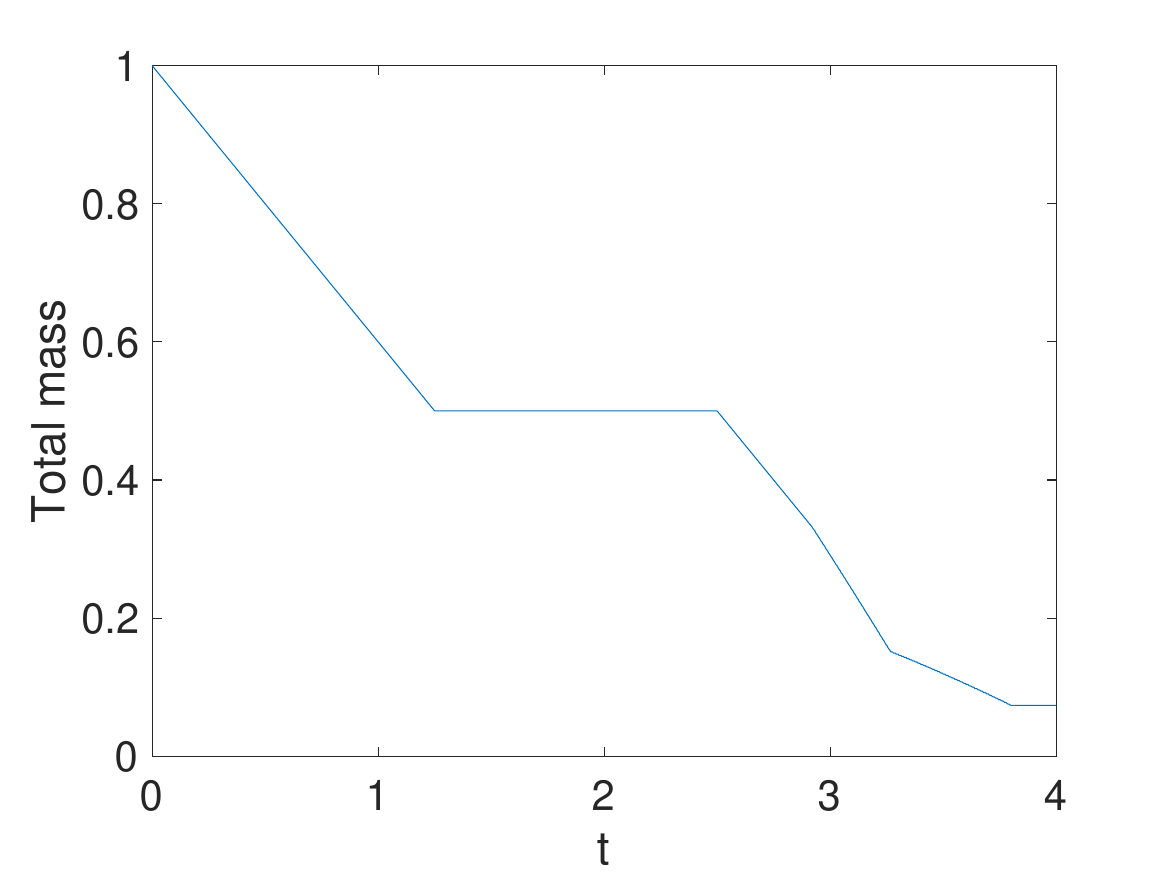}
	\includegraphics[width=0.45\textwidth]{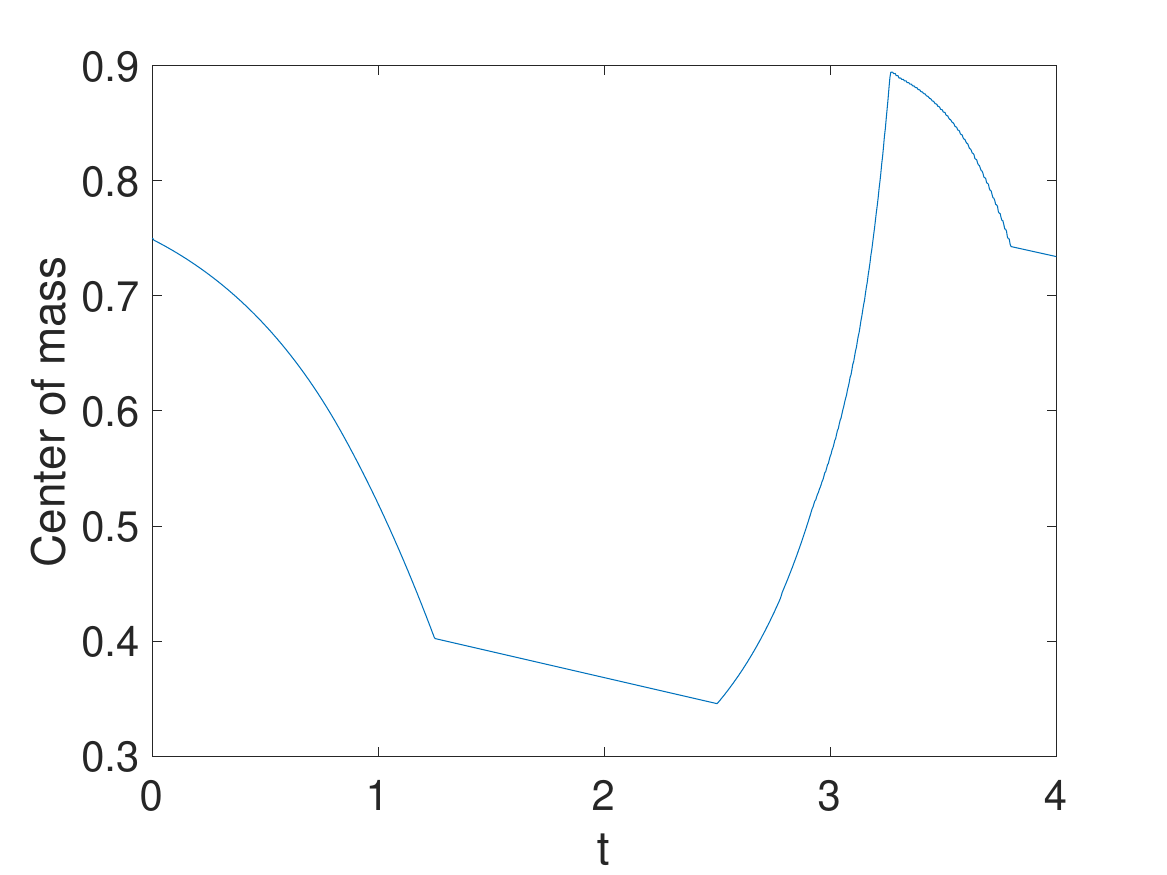}
	\caption{Time evolution of the total mass (left) and time evolution of the center of mass (right).}
	\label{fig:2}
\end{figure}
 
The profile of the center of mass is non-monotone and all the oscillations are coherent with the loss of mass during the time evolution of the population.

It is clear that, in principle, an iterative procedure may be time-consuming.

In particular, the computational time heavily depend both on the number of points of the discretized control set and on the number of numerical particles.
A strategy used in our article has been to work in two steps. In the first step, we have computed the numerical solution by using the full discretized control set and
a reduced number of numerical particles. Then, this solution of the numerical simulation with a relatively low number of numerical particles has been the basis for reducing the control set, by eliminating the elements of the discretized control set which are not relevant for the minimization of the cost fuction.

We have observed, in our numerical experiments, that the fixed point is reached quite rapidly  once the set of admissible control velocities is reduced: 
as shown in Figure 
\ref{fig:3} (left), 3 iterations have been enough to reach the Nash equilibrium in the studied case. In the first iteration, by summing all the modification in
each time step, 1,064,896 individual strategy modifications have been observed; in the second iteration, 8,382 individual modification have been carried
out and in the third one, 1,005 individual strategy modifications have been reported. Finally, in the forth iteration step no individual modifications have been observed, and hence the Nash equilibrium has been reached.
We underline that, in the time span of the simulation, by multiplying the number of time steps by the number of numerical particles,
the total number of possible strategy modifications is equal to $6\times 10^6$.

\begin{figure}[!ht]
	\centering
	\includegraphics[width=0.45\textwidth]{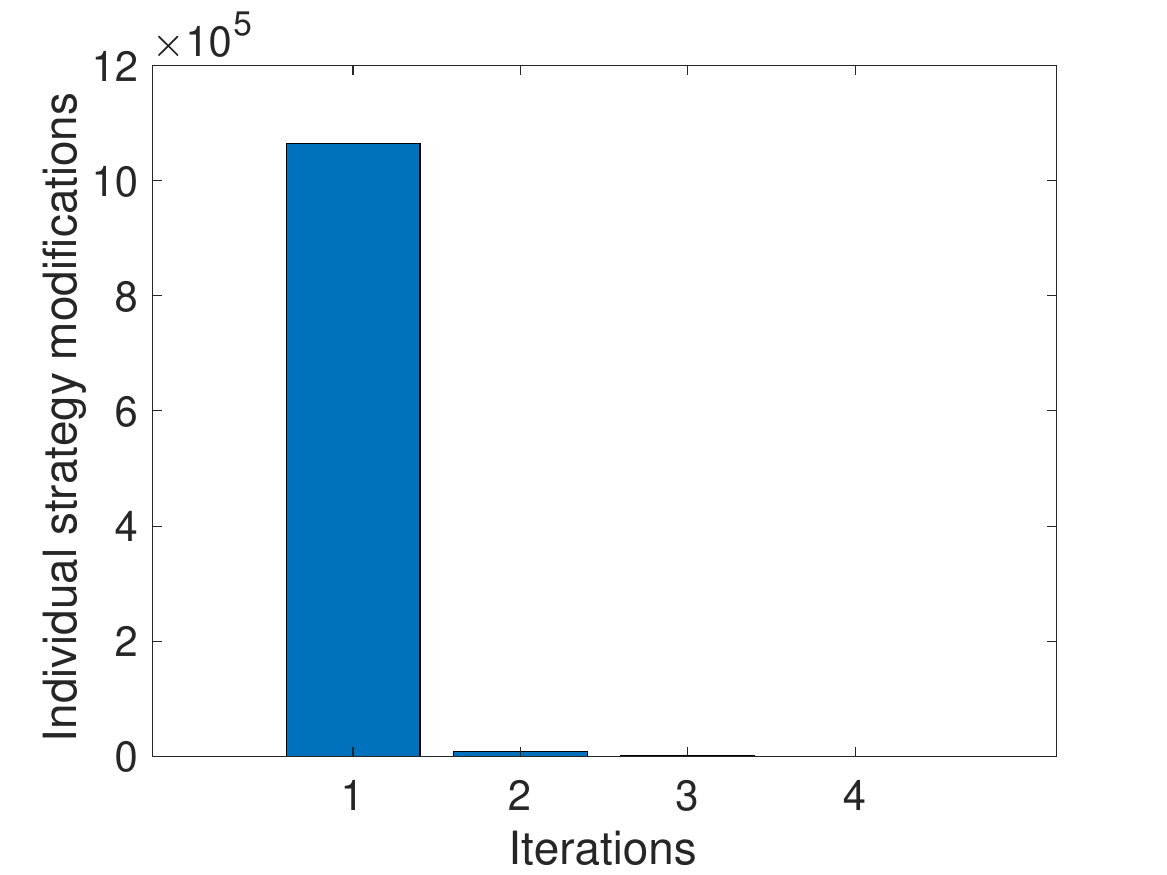}
	\includegraphics[width=0.45\textwidth]{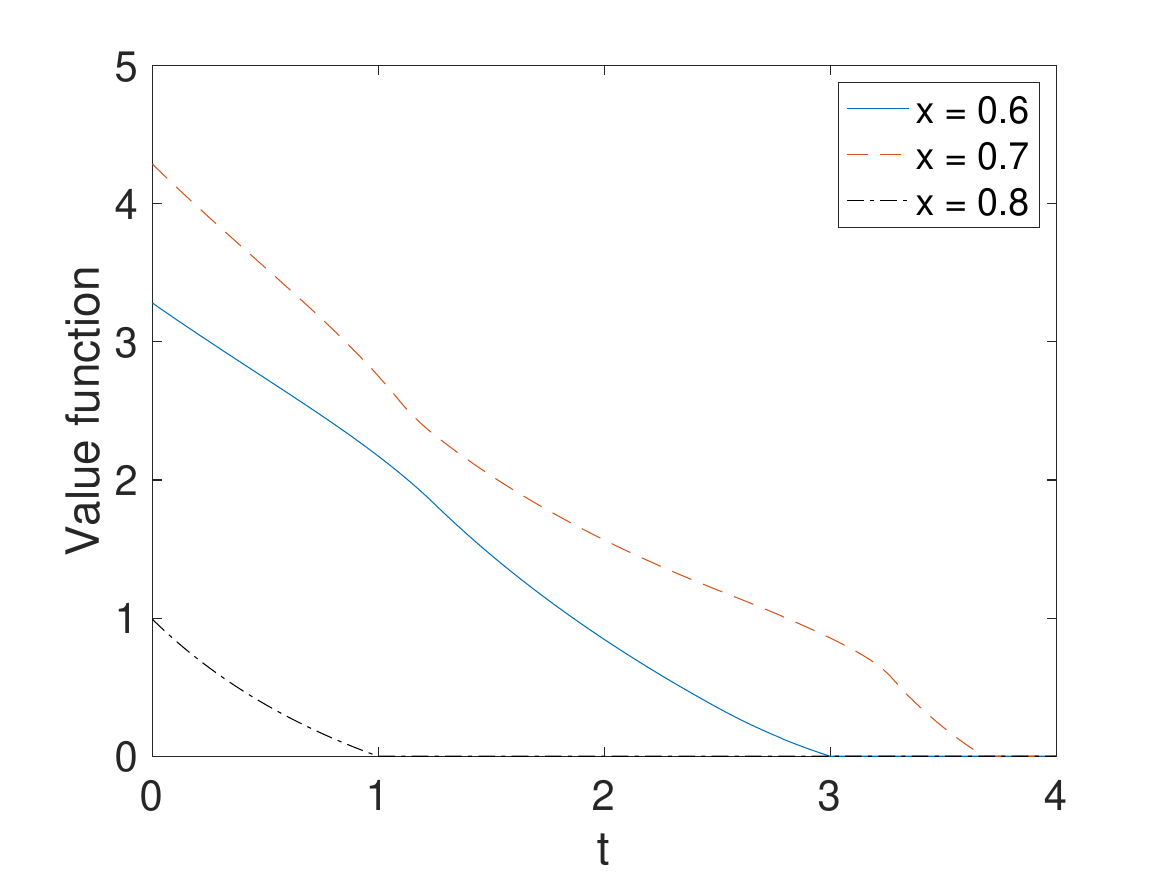}
	\caption{Individual strategy modifications before reaching the Nash equilibrium (left) and time evolution of the value function at $x=0.6$, $x=0.7$
		and $x=0.8$ (right).}
	\label{fig:3}
\end{figure}

We conclude our analysis of the numerical experiment by underlying that our method allows to reconstruct the value function not only as a function of
space and time, but also as a function of time for a particle starting at a given point $x\in\Omega$ as $t=0$. In Figure 
\ref{fig:3} (right), we have plotted the time evolution of the value function for particles starting at $x=0.6$, $x=0.7$ and $x=0.8$. All of the are curves are strictly decreasing functions of time and are identically equal to zero for times greater than the exit time of the corresponding particle (for example,
$t\geq 3,684$ for a particle located at $x=0.7$ when $t=0$).

\section{Appendix} \label{sec:appendix}

We here recall some basic facts concerning H\"older continuous functions and provide a proof of the explicit solution to the product log equation \eqref{eq:productlog}.

\begin{lemma}\label{le:lipbasic}
	Assume that $k:X\rightarrow Y$ is Hö\"older continuous with exponent $\alpha$ and coefficient $\|k\|_{\alpha}$. Then the following hold:
	\begin{enumerate}
		\item if $h:X\rightarrow Y$ is H\"older continuous with exponent $\alpha$ and coefficient $\|h\|_{\alpha}$, then $h + k$ is H\"older continuous with exponent $\alpha$ and coefficient $\|h\|_{\alpha}+\|k\|_{\alpha}$,
		\item if $X$ is bounded and $h:X\rightarrow Y$ is H\"older continuous with exponent $\alpha$ and coefficient $\|h\|_{\alpha}$, then $hk$ is H\"older continuous with exponent $\alpha$ and coefficient $\|k\|_{\infty}\|h\|_{\alpha}+\|h\|_{\infty}\|k\|_{\alpha}$
		\item if $h:X\rightarrow Y$ is H\"older continuous with exponent $\beta$ and coefficient $\|h\|_{\beta}$, then $h \circ k$ is H\"older continuous with exponent $\alpha\beta$ and coefficient $\|h\|_{\beta}\|k\|^{\beta}_{\alpha}$.
		\item if $X$ is bounded and $h:X\rightarrow Y$ is bounded, then $h * k$ is H\"older continuous with exponent $\alpha$ and coefficient $|X|\|h\|_{\infty}\|k\|_{\alpha}$, where $|X|$ denotes the measure of $X$. 
	\end{enumerate}
\end{lemma}
\begin{proof}
	The statement about $f + g$ is trivial.
	
	Assume that $h$ is H\"older continuous with exponent $\beta$ and coefficient $|h|_{\beta}$ and that $k$ is Hö\"older continuous with exponent $\alpha$ and coefficient $|k|_{\alpha}$. Then for every $x_1, x_2 \in X$ we have
	\begin{align}\begin{split}\label{eq:holdprod}
	|h(x_1)k(x_1) - h(x_2)k(x_2)| &\leq |h(x_1)k(x_1) -h(x_1)k(x_2) +h(x_1)k(x_2) - h(x_2)k(x_2)|\\
	&\leq |h(x_1)|k(x_1) -k(x_2)| +|k(x_2)||h(x_1) - h(x_2)|\\
	&\leq \|h\|_{\infty}\|k\|_{\alpha}|x_1 -x_2|^{\alpha} +\|k\|_{\infty}\|h\|_{\beta}|x_1 - x_2|^{\beta},
	\end{split}
	\end{align}
	which proves the statement about $hk$. Concerning $h \circ k$ we have
	\begin{align}\begin{split}\label{eq:holdcirc}
	|h(k(x_1)) - h(k(x_2))| &\leq \|h\|_{\beta}|k(x_1) - k(x_2)|^{\beta}\\
	&\leq \|h\|_{\beta}\|k\|_{\alpha}^{\beta}|x_1 - x_2|^{\alpha\beta}.
	\end{split}
	\end{align}
	Finally,
	\begin{align}\begin{split}\label{eq:convcirc}
	|h*k(x_1) - h*k(x_2)| & = \left|\int_{X} h(y)k(x_1 - y)dy - \int_{X} h(y)k(x_2 - y)dy\right|\\
	&\leq  \int_{X} |h(y)|\left|k(x_1 - y) - k(x_2 - y)\right|dy\\
	&\leq |X|\|h\|_{\infty}\|k\|_{\beta}|x_1 - x_2|^{\beta}.
	\end{split}
	\end{align}
	This concludes the proof. 
\end{proof}

\begin{proposition}\label{prop:productlog}
	Let $p,q\in\R^d$ and $a,b\in\R$. Then the solution of
	\begin{align}\label{eq:productlogeq}
	p + \alpha + a \exp\left(b \alpha \cdot q\right) q = 0
	\end{align}
	is given by
	\begin{align}\label{eq:productlogsol}
	\alpha = -p - a \exp\left(b W\left(ab |q|^2 p\cdot q \exp\left(-bp\cdot q\right) \right) \right)q,
	\end{align}
	where $W$ denotes the principal branch of \mattia{the} Lambert W function \cite{corless1996lambertw}.
\end{proposition}
\begin{proof}
	First notice that if $q = 0$ then the statement is straightforward. Assume therefore that w.l.o.g. $q_1 \ne 0$.
	
	Notice that any $\alpha = (\alpha_1, \ldots, \alpha_d) \in \R^d$ which solves \eqref{eq:productlogeq} must also satisfy the system
	\begin{align*}
	\begin{cases}
	p_1 + \alpha_1 + a q_1 \exp\left(b \alpha \cdot q\right)  = & 0 \\
	\qquad\quad\vdots \notag &  \\
	p_d + \alpha_d + a q_d \exp\left(b \alpha \cdot q\right)  = & 0
	\end{cases}
	\end{align*}
	which implies that for every $i = 2, \ldots, d$ it holds
	\begin{align}\label{eq:linearalpha}
	\alpha_i = \frac{q_i}{q_1}\alpha_1 + \frac{q_i}{q_1}p_1 - p_i.
	\end{align}
	This in turn implies that
	\begin{align*}
	\alpha \cdot q & = \alpha_1q_1 + \sum^d_{i = 2}\left(\frac{q_i}{q_1}\alpha_1 + \frac{q_i}{q_1}p_1 - p_i\right)q_i\\
	& = \alpha_1\frac{1}{q_1}\sum^d_{i = 1}q^2_i + p_1\frac{1}{q_1}\sum^d_{i = 1}q^2_i - \sum^d_{i = 1}p_iq_i\\
	& = \alpha_1 \frac{|q|^2}{q_1} + \frac{p_1}{q_1}|q|^2 -p\cdot q.
	\end{align*}
	If we plug this identity into the equation for $\alpha_1$ we get
	\begin{align*}
	p_1 + \alpha_1 + & a q_1 \exp\left(\alpha_1 \frac{b|q|^2}{q_1} + \frac{bp_1}{q_1}|q|^2 -bp\cdot q\right) = 0 \\
	& \Longleftrightarrow \alpha_1 = -p_1 - a q_1\exp\left(\frac{bp_1}{q_1}|q|^2 -bp\cdot q\right) \exp\left(\alpha_1 \frac{b|q|^2}{q_1}\right)
	\end{align*}
	It is well-known that the solution for the above equation in $\alpha_1$ is given by
	\begin{align*}
	\alpha_1 &= -p_1 - \frac{q_1}{b |q|^2}W\left(ab|q|^2 \exp\left(\frac{bp_1}{q_1}|q|^2 -bp\cdot q\right) \exp\left(-\frac{bp_1}{q_1}|q|^2\right) \right) \\
	& \alpha_1 = -p_1 - \frac{q_1}{b |q|^2}W\left(ab|q|^2 \exp\left(-bp\cdot q\right) \right)\\
	& \alpha_1 = -p_1 - \frac{q_1}{b |q|^2}ab|q|^2 \exp\left(-bp\cdot q\right) \exp\left(W\left(ab|q|^2 \exp\left(-bp\cdot q\right)\right) \right)\\
	& \alpha_1 = -p_1 - a \exp\left(b W\left(ab|q|^2 p\cdot q \exp\left(-bp\cdot q\right) \right) \right)q_1,
	\end{align*}
	where we used the Lambert W identity $W(z) = z\exp\left(-W(z)\right)$. If we plug the above identity into equation \eqref{eq:linearalpha} we get the statement.
\end{proof}


\bigskip
\noindent{\bf Acknowledgements.}
This work has been carried out in the framework of the project \textsl{Kimega} (ANR-14-ACHN-0030-01).
The authors thank Prof. Pierre Cardaliaguet for pointing out the problem studied in the article as well as for useful comments and suggestions.


\bibliography{Literature}	

\end{document}